\documentclass[12pt, reqno]{amsart}
\usepackage{amsmath, amsthm, amscd, amsfonts, amssymb, graphicx, color}
\usepackage[bookmarksnumbered, colorlinks, plainpages]{hyperref}
\usepackage{color}

 \newtheorem{thm}{Theorem}[section]
 
 \newtheorem{lem}[thm]{Lemma}
 \newtheorem{prop}[thm]{Proposition}
 \theoremstyle{definition}
 \newtheorem{defn}[thm]{Definition}
 \theoremstyle{remark}
 \newtheorem{rem}[thm]{Remark}
 \newtheorem*{ex}{Example}
 \numberwithin{equation}{section}

\begin{document}

%
%
%
%
%
%
%
%
%

\title[KSGNS type constructions for $\alpha $-completely positive maps]{KSGNS type constructions for $\alpha$-completely positive maps on Krein $C^*$-modules}

\author{Mohammad Sal Moslehian}
\address{Department of Pure Mathematics\\ Ferdowsi University of Mashhad\\ P.
O. Box 1159\\ Mashhad 91775\\ Iran}
\email{moslehian@um.ac.ir and moslehian@member.ams.org}
\urladdr{http://profsite.um.ac.ir/~moslehian/}

\author{Maria Joi\c{t}a}
\address{$^1$ Department of Mathematics\\ University of Bucharest\\ Bd. Regina
Elisabeta nr. 4-12\\ Bucharest\\ Romania}
\email{mjoita@fmi.unibuc.ro}
\urladdr{http://sites.google.com/a/g.unibuc.ro/maria-joita/}

\author{Un Cig Ji}
\address{Department of Mathematics\\ Research Institute of
Mathematical Finance\\ Chungbuk National University\\ Cheongju 361-763\\ Korea}
\email{uncigji@chungbuk.ac.kr}




\begin{abstract}
In this paper, we investigate $\Phi$-maps associated to a certain type of $\alpha$-completely positive maps. We then prove a KSGNS (Kasparov--Stinespring--Gel'fand--Naimark--Segal) type
theorem for $\alpha $-completely positive maps on Krein $C^*$-modules and show that the minimal KSGNS construction is unique up to
unitary equivalence. We also establish a covariant version of the KSGNS type
theorem for a covariant $\alpha $-completely positive map and study the structure of minimal covariant KSGNS constructions.
\end{abstract}

\maketitle

\section{Introduction and Preliminaries}

Throughout the paper by a map (an operator) we mean a bounded linear one.

\subsection{Completely Positive maps}

Completely positive maps can be regarded as extensions of states,
representations and conditional expectations. A (not necessarily linear) map between $C^*$-algebras, which sends positive elements to positive elements is called positive. By a completely positive map we mean a
map $\varphi: \mathcal{A} \to \mathcal{B}$ between $C^*$-algebras with the property that for each $n$, $\varphi$ is $n$-positive, in the sense that, the map $\varphi_n$ from the $C^*$-algebra $\mathcal{M}_n(\mathcal{A})$ of
all $n\times n$ matrices with entries in $\mathcal{A}$ into $\mathcal{M}_n(\mathcal{B})$ defined by $\varphi([a_{ij}])=[\varphi(a_{ij})]$ is positive. For example, any positive map $\varphi: C(X)\to C(Y)$ between unital $C^*$-algebras is completely positive. There are positive maps which are not completely positive, for instance, the map $\varphi: \mathcal{M}_3(\mathbb{C}) \to \mathcal{M}_3(\mathbb{C})$ defined by $\varphi(X)=2{\rm tr}(X)I- X$ is $2$-positive but not $3$-positive.

In quantum information theory, a quantum operation is defined as a certain completely positive linear map and plays an essential role in describing the transformations, which a quantum mechanical system may undergo. The structure of completely positive maps, the description of
the order relation on their set, the characterization of pure maps and
extremal maps according to their structure are significants in understanding a lot of problems, see \cite{MSM}. \\
The Stinespring theorem for $C^*$-algebras is a significant
generalization of the Gel'fand--Naimark--Segal (GNS) construction to
operator-valued maps. The GNS construction gives a correspondence between cyclic $*$-representations of a $C^*$-algebra and its certain linear functionals and used to establish the celebrated Gel'fand-–Naimark--Segal theorem, which characterize $C^*$-algebras as algebras of Hilbert space operators, see \cite{DOR}. The classical version of Stinespring's theorem states that a map $T$ from a unital $%
C^*$-algebra $\mathcal{A}$ into the $C^*$-algebra $\mathcal{L}(%
\mathcal{H})$ of all (bounded linear) operators is completely positive if
and only if it is of the form $T(a) = V^*\pi(a)V\,\,(a\in \mathcal{A})$,
where $\pi: \mathcal{A}\to \mathcal{L}(K)$ is a $*$-representation of $%
\mathcal{A}$ and $V: \mathcal{H} \to \mathcal{K}$ is a map. Kasparov \cite{Ka} extended the Stinespring theorem for
completely positive linear maps from a $C^*$-algebra $\mathcal{A}$ to the $C^*$-algebra of all adjointable operators on the Hilbert $C^*$-module $\mathcal{H_{B}}\ $over $C^*$-algebra $\mathcal{B}$. He showed that a completely positive linear map $\varphi $ from $\mathcal{A}$ to $\mathcal{H_{B}}\ $ is of the form $\varphi (a)=V^*\pi (a)V(a\in %
\mathcal{A})$, where $\pi :\mathcal{A}\to \mathcal{L}(\mathcal{X})$
is a $\ast $-representation of $\mathcal{A}$ on a Hilbert $C^*$-module
$\mathcal{X}$ and $V:\mathcal{H_{B}}\to \mathcal{A}$ is a bounded
linear map. Bhat et al. \cite{BRS} extended the Stinespring
theorem for completely positive $\varphi $-maps on Hilbert $C^*$-modules. They provided a Stinespring construction associated to a completely positive $\varphi $-map $\Phi $ on a Hilbert $C^*$-module $X$ in terms of the Stinespring construction associated to the underlying completely positive map $\varphi $. A covariant version of this construction
can be found in \cite{J}. Several generalizations of Stinespring theorem are given by mathematicians; see \cite{ACM} and references therein.

On the other hand, the notion of locality in the Wightman formulation of gauge quantum field theory \cite{St,JS} conflicts with the notion of positivity. To avoid this, Jakobczyk and Strocchi \cite{JS} introduced the concept of $\alpha $-positivity. A typical example of an $\alpha$-positive map is as follows, cf. \cite{AO}.
\begin{ex}\label{alpha11}
Let $\mathcal{B}\subseteq \mathcal{A}$ be unital $C^*$-algebras and $P:\mathcal{A}\to\mathcal{B}$ be a conditional expectation, i.e. a unital $*$-map satisfying $P(b_1ab_2)=b_1P(a)b_2\,\, (a\in\mathcal{A}, b_1,b_2\in \mathcal{B})$. Let $\rho$ be an $P$-functional, i.e. a Hermitian functional on $\mathcal{A}$ such that $\rho(P(a))=\rho(a)$ and $2\rho(P(a)^*P(a))\geq \rho(a^*a)$ for all $a\in\mathcal{A}$. Then $\alpha(a):=2P(a)-a$ is a linear involution, i.e. $\alpha^2(a)=a$ for all $a\in \mathcal{A}$ and satisfies $\rho(\alpha(a_1)\alpha(a_2))=\rho(a_1a_2)$ and $\rho(\alpha(a)^*a)\geq 0$ for all $a_1, a_2, a \in \mathcal{A}$.
\end{ex}
Motivated by the notion of $\alpha $-positivity and $P$-functionals (see \cite{H, AO}), Heo et al. \cite{HHJ} introduced the notion of $\alpha$-completely positive map between $C^*$-algebras and provided a Kasparov--Stinespring--Gelfand--Naimark--Segal
(KSGNS) type construction for $\alpha $-completely positive maps. Here, positivity is inherent in Hermitian maps in terms of the map $\alpha $. The $\alpha $-complete positivity provides a positive definite inner product
associated to the indefinite one, and the interplay between these two is indeed the characteristic feature of Krein spaces among all indefinite inner product spaces.

\subsection{Hilbert $C^*$-moduels}

Hilbert $C^{*}$-modules are essentially objects like Hilbert spaces, except that the inner product, instead of being complex--valued, takes its values in a $C^{*}$-algebra. Although Hilbert $C^{*}$-modules behave like Hilbert spaces in some ways, some fundamental Hilbert space properties like Pythagoras equality, the adjointability of operators and the decomposition into orthogonal complements do not hold in general.
An inner product $C^*$-module over a $C^{*}$-algebra $\mathcal A$ is a complex linear space $\mathcal{X}$ which is a right $\mathcal A$-module with a compatible scalar multiplication (i.e., $\gamma(xa) = (\gamma x)a = x(\gamma a)$ for all $x\in \mathcal{X}, a\in\mathcal A, \gamma\in\mathbb{C}$) and equipped with an $\mathcal A$-valued inner product $\langle\cdot,\cdot\rangle \,: \mathcal{X}\times \mathcal{X}\longrightarrow\mathcal A$ satisfying\\
(i) $\langle x, \gamma y + \mu z\rangle = \gamma\langle x, y\rangle + \mu\langle x, z\rangle$,\\
(ii) $\langle x, ya\rangle=\langle x, y\rangle a$,\\
(iii) $\langle x, y\rangle^*=\langle y, x\rangle$,\\
(iv) $\langle x, x\rangle\geq0$ and $\langle x, x\rangle=0$ if and only if $x=0$,\\
for all $x, y, z\in \mathcal{X}, a\in\mathcal A, \gamma, \mu\in\mathbb{C}$. It is easy to observe that $\|x\|=\|\langle x, x\rangle\|^{\frac{1}{2}}$ defines a norm on $\mathcal{X}$, where the later
norm is that of $\mathcal{A}$. If $\mathcal{X}$ with respect to this norm is complete, then it is called a \emph{Hilbert $\mathcal A$-module}, or a \emph{Hilbert $C^*$-module} over $\mathcal A$. Complex Hilbert spaces are (left) Hilbert $\mathbb{C}$-modules. Any $C^*$-algebra $\mathcal A$ can be regarded as a Hilbert $C^*$-module over itself via $\langle a, b\rangle :=a^* b$. The Hilbert $\mathcal{A}$-module $\mathcal{X}$
is called full if the ideal generated by $\{\langle x,y\rangle: x, y\in \mathcal{X}\}$ is dense in $\mathcal{A}$. A map $T$ from a Hilbert $\mathcal{A}$-module $\mathcal{X}$ into another Hilbert $\mathcal{A}$-module $\mathcal{Y}$
is adjointable if there is a map $T^*:\mathcal{Y}%
\to \mathcal{X}$ such that $\langle Tx,y\rangle =\langle x,T^*y\rangle \ $for all $x\in \mathcal{X}$ and $y\in $ $\mathcal{Y}$. It is easy to show that any adjointable map $T$ is a bounded $\mathcal{A}$-module one. We denote by $\mathcal{L}(\mathcal{X},\mathcal{Y})$ the space of all adjointable
module morphisms from $\mathcal{X}$ to $\mathcal{Y}$. In the case that $%
\mathcal{X}=$ $\mathcal{Y}$, it is denoted by $\mathcal{L}(\mathcal{X})$,
which is a $C^*$-algebra. The reader is referred to \cite{LAN} for
basic notions related to Hilbert $C^*$-modules.\\
Let $\varphi: \mathcal{A}\to \mathcal{B}$ be a linear map. A map $\Phi $ from a Hilbert $\mathcal{A}$-module $\mathcal{X}$ into another Hilbert $\mathcal{B}$-module $\mathcal{Y}$ is called a $\varphi $-map if $\langle \Phi (x),\Phi (y)\rangle =\varphi (\langle x,y\rangle )$ for all $x,y\in \mathcal{X}$. A $\varphi $-map $\Phi $ is completely positive if $\varphi $ is completely positive.

\subsection{Krein spaces}

A Krein space, as an indefinite generalization of a Hilbert space, is a vector space equipped with a symmetric or Hermitian bilinear form $[\cdot, \cdot]$ in such a way that $[x, x]$ can be positive, negative or zero, \cite{BOG, Azizov} for more information. This notion was first defined by Ginzburg \cite{Gi}. Indeed lack of positivity in some models in quantum field theories made theoretical physicists to consider indefinite structures. Since then many mathematicians have investigated them. Recently Heo et al. \cite{HHJ, HJK} studied Krein $C^*$-modules and covariant representations on Krein $C^*$-modules. A treatment of operator convex functions is presented in \cite{MD}.

Motivated by the classical theory of Krein spaces, we can introduce a parallel theory to the setting of Hilbert $C^*$-modules. For an exposition on the subject see \cite{Ando}.

\begin{defn}
\upshape
Let $(\mathcal{H},\langle\cdot,\cdot\rangle)$ be a Hilbert $\mathcal{A}$%
-module over a $C^*$-algebra $\mathcal{A}$ and let $J$ be a fundamental
symmetry on $\mathcal{H}$, i.e., $J=J^*=J^{-1}$. Then one can define an
indefinite $\mathcal{A}$-valued inner product by
\begin{eqnarray*}
[x,y]: =\langle Jx,y\rangle\qquad(x,y\in \mathcal{H}).
\end{eqnarray*}
In this case $(\mathcal{H},J)$ is called a Krein $\mathcal{A}$-module.
\end{defn}

For a Krein $\mathcal{A}$-module $(\mathcal{H},J)$, if $\mathcal{H}_+$, $%
\mathcal{H}_-$ are the ranges of projections $P_+=(I+J)/2$, $P_-=(I-J)/2$,
where $I$ denote the identity operator, then one obtains the orthogonal
direct sum $\mathcal{H}=\mathcal{H}_+\oplus \mathcal{H}_-$, $J=P_+-P_-$, $%
[x,x]=|P_+x|^2-|P_-x|^2$ and $\langle x,x \rangle=|P_+x|^2+|P_-x|^2$. If $%
\mathcal{A}=\mathbb{C}$, then we reach the classical definition of a Krein
space and its fundamental structure. Obviously, if $J=I$, then the theory
reduces to the theory of Hilbert spaces.

\begin{defn}
\upshape
Let $\mathcal{A}$ be a $C^*$-algebra and $\alpha$ be a $*$-automorphism
such that $\alpha^2$ is the identity operator. Evidently if $\mathcal{A}$ is
unital, then $\alpha(1)=1$. One can define an indefinite involution $%
x^\#=\alpha(x^*)$ on $\mathcal{A}$. Then $(\mathcal{A},\alpha)$ is called a
Krein $C^*$-algebra. Thus $\|\alpha(x^\#)x\|=\|x\|^2$ for all $x \in %
\mathcal{A}$. It is easy to see that $\mathcal{A} = \mathcal{A}_+ \oplus %
\mathcal{A}_-$, where $\mathcal{A}_+ = \{x \in \mathcal{A} | \alpha(x) = x\}$
and $\mathcal{A}_- = \{x \in \mathcal{A} | \alpha(x) = -x\}$. In addition, $%
\mathcal{A}_+$ is a $C^*$-algebra and $(\mathcal{A}_-,\alpha)$ is a Krein $C^*$-module over $\mathcal{A}_+$
with $[x,y]=x^\# y=\alpha(x^*)y=-x^*y$ for all $x,y\in%
\mathcal{A}_-$.
\end{defn}


\begin{ex}
\upshape
Consider the usual unital commutative $C^*$-algebra $C[0,1]$ and the
automorphism $\alpha(f)(x)=f(1-x)$. Then $C[0,1]$ together with the
indefinite involution $f^\#(x)=\overline{f(1-x)}$ is a Krein $C^*$%
-algebra.
\end{ex}


\begin{ex}
\upshape
Let $(\mathcal{H},J)$ be a Krein $C^*$-module. For each $T \in \mathcal{%
L}(\mathcal{H})$ there exists an operator $T^\# \in \mathcal{L}(\mathcal{H})$
such that $[T\xi,\eta]=[\xi, T^\#\eta]$. Evidently $T^\#=JT^*J$ and $%
\mathcal{L}(\mathcal{H})$ equipped with $\alpha(T)=JTJ$ is a Krein $C^*$%
-algebra.
\end{ex}


\begin{defn}
\upshape Let $\mathcal{A}$ be a $C^*$-algebra and let $(\mathcal{K},J)$
be a Krein $C^*$-module. A homomorphism $\pi :\mathcal{A}\to
\mathcal{L}(\mathcal{K})$ is called a representation of $\mathcal{A}$ on $(%
\mathcal{K},J)$ if $\pi (a^*)=J\pi (a)^*J=\pi (a)^{\#}$, or
equivalently, $[\pi (a)\xi ,\eta ]=[\xi ,\pi (a^*)\eta ]$.
\end{defn}

Focusing on the structure of indefinite version of Hilbert $C^*$-modules, we investigate $\Phi $-maps associated to $\varphi $-maps, which are a certain type of $\alpha $-completely positive maps. We then prove a KSGNS (Kasparov--Stinespring--Gel'fand--Naimark--Segal) type theorem for $\alpha $-completely positive maps on Krein $C^*$-modules and show that the minimal KSGNS construction is unique up to unitary equivalence. We also establish a covariant version of the KSGNS type theorem for a covariant $\alpha $-completely positive map and study the structure of minimal covariant KSGNS constructions. Our result provide some variants and some generalizations of results of \cite{HJK} in the context of maps on Krein $C^*$-modules.

\section{KSGNS type construction for $\protect\alpha $-CP maps}

In this section, we assume that $(\mathcal{A},\alpha )$ is a unital $C^*
$-algebra with the unit $1$. We start our work with the following modified
definition of \cite[Definition 2.4]{HHJ} playing an essential role in the
paper. It provides a generalization of $\alpha$-positivity introduced in Example \ref{alpha11}.

\begin{defn}
\label{def4} \upshape An $\alpha $-completely positive map (briefly, $\alpha
$-CP) of $\mathcal{A}$ on a Krein $C^*$-module $\left( \mathcal{H}%
,J\right) $ is a $*$-map $\varphi :\mathcal{A}\to \mathcal{L}(%
\mathcal{H})$ such that

\begin{enumerate}
\item[(i)] $\varphi (a^{\#})=\varphi (a)^{\#}=\varphi (a^*)$, or
equivalently, $\varphi \left( \alpha (a)\right) =J\varphi (a)J=\varphi (a)$
for all $a\in \mathcal{A}$;

\item[(ii)] the $n\times n$ matrix $[\varphi (a_{i}^{\#}a_{j})]$ is positive
for all $n\geq 1$ and each $a_{1},\cdots ,a_{n}\in \mathcal{A}$, or
equivalently, $\sum_{i=1}^{n}\sum_{j=1}^{n}\left\langle \xi
_{i},\varphi \left( \alpha \left( a_{i}\right)^*a_{j}\right) \xi
_{j}\right\rangle \geq 0$ for all $n\geq 1$, $a_{1},\cdots ,a_{n}\in %
\mathcal{A}$ and $\xi _{1},\cdots ,\xi _{n}\in \mathcal{H}$;

\item[(iii)] for any $a\in \mathcal{A}$, there is $M(a)>0$ such that
\begin{equation*}
\left[ \varphi \left( \left( aa_{i}\right) ^{\#}aa_{j}\right) \right]
_{i,j=1}^{n}\leq M(a)\left[ \varphi \left( a_{i}^{\#}a_{j}\right) \right]
_{i,j=1}^{n}
\end{equation*}%
for all $n\geq 1$ and $a_{1},\cdots ,a_{n}\in \mathcal{A}$.
\end{enumerate}

To be sure that our maps are continuous, we may assume that the constant $%
M(a)$ is of the form $K(a)\left\Vert a\right\Vert $ with $K(a)>0$
\end{defn}

Let $\left( \mathcal{H}_{1},J_{1}\right) $ and $\left( \mathcal{H}%
_{2},J_{2}\right) $ be Krein $C^*$-modules. For $T\in \mathcal{L}(%
\mathcal{H}_{1},\mathcal{H}_{2})$, let us put
\[
T^{\#}=J_{1}T^*J_{2}.
\]


\begin{defn}
\upshape
Let $\mathcal{X}$ be a Hilbert $\mathcal{A}$-module and let $(\mathcal{H}%
_{1},J_1),(\mathcal{H}_{2},J_2)$ Krein $\mathcal{B}$-modules. For an $\alpha$%
-CP map $\varphi:\mathcal{A}\to \mathcal{L}(\mathcal{H}_{1})$, a map
$\Phi :\mathcal{X}\to \mathcal{L}(\mathcal{H}_{1},\mathcal{H}_{2})$
is called a ($\alpha $-completely positive) $\varphi$-map if for any $x,y\in %
\mathcal{X}$,
\[
\Phi\left(x\right)^{\#} \Phi\left(y\right) =\varphi \left(\left\langle
x,y\right\rangle \right).
\]
\end{defn}


\begin{defn}
\upshape
A representation of a Hilbert $\mathcal{A}$-module $\mathcal{X}\ $on Krein $%
\mathcal{B}$-modules $\left(\mathcal{H}_{1},J_{1}\right) $ and $\left( %
\mathcal{H}_{2},J_{2}\right) \ $is a map $\pi _{\mathcal{X}}:\mathcal{X}%
\to \mathcal{L}(\mathcal{H}_{1},\mathcal{H}_{2})$ with the property
that there is a representation $\pi _{\mathcal{A}}$\ of $\mathcal{A}$ on $%
\left( \mathcal{H}_{1},J_{1}\right) $ such that
\begin{equation*}
\pi _{\mathcal{X}}\left( x\right) ^{\#}\pi _{\mathcal{X}}\left( y\right)
=\pi _{\mathcal{A}}\left( \left\langle x,y\right\rangle \right).
\end{equation*}%
Then we say that $\pi _{\mathcal{X}}$ is a $\pi _{\mathcal{A}}$%
-representation.
\end{defn}


\begin{rem}
\upshape
Let $\pi _{\mathcal{X}}$ be a $\pi _{\mathcal{A}}$-representation of a
Hilbert $\mathcal{A}$-module $\mathcal{X}$ on Krein $\mathcal{B}$-modules $%
\left( \mathcal{H}_{1},J_{1}\right) $ and $\left( \mathcal{H}%
_{2},J_{2}\right)$.

\begin{enumerate}
\item If $\mathcal{X}$ is full, then $\pi _{\mathcal{A}}$ is unique.

\item If $\left[ \pi _{\mathcal{X}}\left( \mathcal{X}\right) \mathcal{H}_{1}%
\right] =\mathcal{H}_{2},$ then for any $x\in\mathcal{X}$ and $a\in %
\mathcal{A}$,
\begin{equation} \label{eq2.3}
\pi _{\mathcal{X}}\left(xa\right) =\pi _{\mathcal{X}}\left(x\right) \pi_{%
\mathcal{A}}(a).
\end{equation}
Indeed, for each $x,y\in \mathcal{X}$ and $a\in \mathcal{A},$ we obtain that
\begin{align*}
\pi _{\mathcal{X}}\left( xa\right)^{\#}\pi _{\mathcal{X}}\left( y\right)
&=\pi _{\mathcal{A}}\left( \left\langle xa,y\right\rangle \right) =\pi _{%
\mathcal{A}}\left(a^*\right) \pi _{\mathcal{A}}\left(\left\langle
x,y\right\rangle \right) \\
&=\pi _{\mathcal{A}}\left(a\right)^{\#}\pi _{\mathcal{X}}\left(
x\right)^{\#}\pi _{\mathcal{X}}\left( y\right) \\
&=\left(\pi _{\mathcal{X}}\left(x\right)\pi _{\mathcal{A}}(a)\right)^{\#}%
\pi_{\mathcal{X}}\left( y\right),
\end{align*}
whence we deduce that $\pi _{\mathcal{X}}\left( xa\right) ^{\#}=\left(\pi _{%
\mathcal{X}}\left(x\right)\pi _{\mathcal{A}}(a)\right)^{\#}$, which implies %
\eqref{eq2.3}.
\end{enumerate}
\end{rem}


\noindent The next proposition gives a typical example of an $\alpha$-completely
positive map.

\begin{prop}
Let $(\mathcal{A},\alpha )$ be a unital Krein $C^*$-algebra, $\pi _{%
\mathcal{A}}$ be a representation of $\mathcal{A}$ on Krein $C^*$%
-modules $(\mathcal{K}_{1},J_{1}=\mathrm{id}_{\mathcal{K}_{1}})$, let $\pi _{%
\mathcal{X}}$ be a $\pi _{\mathcal{A}}$-representation of a Hilbert $\mathcal{A}$-module $\mathcal{X}$ on $(%
\mathcal{K}_{1},J_{1}=\mathrm{id}_{\mathcal{K}_{1}})$ and $(\mathcal{K}%
_{2},J_{2}=\mathrm{id}_{\mathcal{K}_{2}})$, let $\left( \mathcal{H}%
_{1},J_{3}\right) $ and $\left( \mathcal{H}_{2},J_{4}=\mathrm{id}_{%
\mathcal{H}_{2}}\right) $ be Krein $C^*$-modules, and $V:\mathcal{H}%
_{1}\to $ $\mathcal{K}_{1}$ and $W:\mathcal{H}_{2}\to %
\mathcal{K}_{2}$ be two operators such that $V^{\#}=V^*,\pi _{%
\mathcal{A}}(\alpha (a))V=J_{1}\pi _{\mathcal{A}}(a)VJ_{3}$ for all $a\in %
\mathcal{A}$, and finally let $W$ be a coisometry with $W^{\#}=W^*$.
Then the map $\varphi :\mathcal{A}\to \mathcal{L}\left( \mathcal{H}%
_{1}\right) $ given by
\begin{equation*}
\varphi (a)=V^{\#}\pi _{\mathcal{A}}(a)V\ \ \ \left( a\in \mathcal{A}\right)
\end{equation*}%
is an $\alpha $-completely positive map and the map $\Phi :\mathcal{X}%
\to \mathcal{L}(\mathcal{H}_{1},\mathcal{H}_{2})$ given by
\begin{equation*}
\Phi (x)=W^{\#}\pi _{\mathcal{X}}(x)V\ \ \left( x\in \mathcal{X}\right)
\end{equation*}%
is an $\alpha $-completely positive $\varphi $-map.
\end{prop}

\begin{proof}
Indeed, we have
\begin{equation*}
\varphi \left( a^*\right) =V^{\#}\pi _{\mathcal{A}}(a^*)V=V^{\#}\pi _{\mathcal{A}}(a)^*V=\left( V^{\#}\pi _{\mathcal{A}%
}(a)V\right)^*=\varphi \left( a\right)^*
\end{equation*}
for all $a\in \mathcal{A}$ and
\begin{equation*}
\varphi (\alpha \left( a\right) )=V^{\#}\pi _{\mathcal{A}}(\alpha \left(
a\right) )V\ =V^{\#}J_{1}\pi _{\mathcal{A}}(a)VJ_{3}=J_{3}V^{\#}\pi _{%
\mathcal{A}}(a)VJ_{3}=J_{3}\varphi (a)J_{3}
\end{equation*}%
for all $a\in \mathcal{A}$. On the other hand,
\begin{equation*}
\varphi (\alpha \left( a\right) )=V^{\#}\pi _{\mathcal{A}}(\alpha \left(
a\right) )V\ =V^{\#}\pi _{\mathcal{A}}(a)J_{1}VJ_{3}=V^{\#}\pi _{\mathcal{A}%
}(a)V=\varphi (a)
\end{equation*}%
for all $a\in \mathcal{A}.$Therefore, $\varphi (\alpha \left( a\right)
)=J_{3}\varphi (a)J_{3}=\varphi (a)$ for all $a\in \mathcal{A}.$ Now, let $%
a_{1},\cdots ,a_{n}\in \mathcal{A}$ and $\xi _{1},\cdots ,\xi _{n}\in %
\mathcal{H}_{1}$.
Then we have
\begin{align*}
\sum_{i=1}^{n}\sum_{j=1}^{n}\left\langle \xi _{i},\varphi
\left( \alpha \left( a_{i}\right)^*a_{j}\right) \xi _{j}\right\rangle
&=\sum_{i=1}^{n}\sum_{j=1}^{n}\left\langle \xi
_{i},V^{\#}\pi _{\mathcal{A}}(\alpha \left( a_{i}\right)^*a_{j})V\xi
_{j}\right\rangle \\
&=\sum_{i=1}^{n}\sum_{j=1}^{n}\left\langle \pi _{\mathcal{A}%
}(\alpha \left( a_{i}\right)^*)^*V\xi _{i},\pi _{\mathcal{A}%
}(a_{j})V\xi _{j}\right\rangle \\
&=\sum_{i=1}^{n}\sum_{j=1}^{n}\left\langle J_{1}\pi _{%
\mathcal{A}}(\alpha \left( a_{i}\right) )J_{1}V\xi _{i},\pi _{\mathcal{A}%
}(a_{j})V\xi _{j}\right\rangle \\
&=\sum_{i=1}^{n}\sum_{j=1}^{n}\left\langle \pi _{\mathcal{A}%
}(a_{i})VJ_{3}\xi _{i},\pi _{\mathcal{A}}(a_{j})V\xi _{j}\right\rangle \\
&=\left\langle \sum_{i=1}^{n}\pi _{\mathcal{A}}(a_{i})V\xi
_{i},\sum_{j=1}^{n}\pi _{\mathcal{A}}(a_{j})V\xi _{j}\right\rangle
\geq 0.
\end{align*}%
Let all $n\geq 1$ and $a,a_{1},\cdots ,a_{n}\in \mathcal{A}$. Then
\begin{align*}
&\hspace{-1.5cm}\left\langle \left[ \varphi \left( \left( aa_{i}\right)
^{\#}aa_{j}\right) \right] _{i,j=1}^{n}\left( \xi _{k}\right)
_{k=1}^{n},\left( \xi _{k}\right) _{k=1}^{n}\right\rangle \\
&=\left\langle \left( \sum_{j=1}^{n}V^{\#}\pi _{\mathcal{A}}(\alpha
\left( a_{i}^*a^*\right) aa_{j})V\xi _{j}\right)
_{i=1}^{n},\left( \xi _{k}\right) _{k=1}^{n}\right\rangle \\
&=\ \sum_{i=1}^{n}\left\langle \sum_{j=1}^{n}V^{\#}\pi _{%
\mathcal{A}}(\alpha \left( a_{i}^*a^*\right) aa_{j})V\xi
_{j},\xi _{i}\right\rangle \\
&=\ \sum_{i=1}^{n}\sum_{j=1}^{n}\left\langle \pi _{%
\mathcal{A}}(\alpha \left( a^*\right) aa_{j})V\xi _{j},\pi _{%
\mathcal{A}}(\alpha \left( a_{i}^*\right) )^*V\xi
_{i}\right\rangle \\
&=\left\langle \pi _{\mathcal{A}}(\alpha \left( a^*\right)
a)\sum_{j=1}^{n}\pi _{\mathcal{A}}\left( a_{j}\right) V\xi
_{j},\sum_{i=1}^{n}\pi _{\mathcal{A}}(\alpha \left( a_{i}\right)
)V\xi _{i}\right\rangle \\
&=\left\langle \pi _{\mathcal{A}}(\alpha \left( a^*\right)
a)\sum_{j=1}^{n}\pi _{\mathcal{A}}\left( a_{j}\right) V\xi
_{j},\sum_{i=1}^{n}J_{1}\pi _{\mathcal{A}}(a_{i})VJ_{3}\xi
_{i}\right\rangle \\
&=\left\langle \pi _{\mathcal{A}}(\alpha \left( a^*\right)
a)\sum_{j=1}^{n}\pi _{\mathcal{A}}\left( a_{j}\right) V\xi
_{j},\sum_{i=1}^{n}{}_{1}\pi _{\mathcal{A}}(a_{i})V\xi
_{i}\right\rangle \\
&\leq \left\Vert \pi _{\mathcal{A}}(\alpha \left( a^*\right)
a)\right\Vert \left\langle \sum_{j=1}^{n}\pi _{\mathcal{A}}\left(
a_{j}\right) V\xi _{j},\sum_{i=1}^{n}{}_{1}\pi _{\mathcal{A}%
}(a_{i})V\xi _{i}\right\rangle \\
&\leq \left\Vert \pi _{\mathcal{A}}(\alpha \left( a^*\right)
a)\right\Vert \left\langle \left[ \varphi \left( \left( a_{i}\right)
^{\#}a_{j}\right) \right] _{i,j=1}^{n}\left( \xi _{k}\right)
_{k=1}^{n},\left( \xi _{k}\right) _{k=1}^{n}\right\rangle
\end{align*}%
for all $\xi _{1,}\xi _{2},\cdots ,\xi _{n}\in \mathcal{H}_{1}$. Thus we
showed that $\varphi $ is an $\alpha $-completely positive map.

To show that $\Phi $ is an $\alpha $-completely positive $\varphi $-map, let
$x,y\in $ $\mathcal{A}$. We have
\begin{eqnarray*}
\Phi (x)^{\#}\Phi (y) &=J_{3}V^{\#}\pi _{\mathcal{X}}(x)^*WJ_{4}W^{\#}\pi _{\mathcal{X}}(x)V=V^{\#}J_{1}\pi _{\mathcal{X}%
}(\left\langle x,y\right\rangle )V \\
&=V^{\#}\pi _{\mathcal{X}}(\left\langle x,y\right\rangle )V=\varphi \left(
\left\langle x,y\right\rangle \right).
\end{eqnarray*}
\end{proof} 

From now on we assume that $(\mathcal{H}_1,J_1)$ and $(\mathcal{H}_2,J_2=%
\mathrm{id}_{\mathcal{H}_2})$ are Krein $\mathcal{B}$-modules, and $\varphi:%
\mathcal{A}\to\mathcal{L}(\mathcal{H}_{1})$ is an $\alpha $-CP map.
The next Theorem can be regarded as a generalization in the context of maps on Krein $C^*$-modules of \cite[Theorem 4.4]{HHJ}.

\begin{thm}
\label{main1} Let $\mathcal{X}$ be a Hilbert $\mathcal{A}$-module and $\Phi:%
\mathcal{X}\to \mathcal{L}(\mathcal{H}_{1},\mathcal{H}_{2})$ a $%
\varphi$-map. Then there are a Krein $\mathcal{B}$-module $\left( \mathcal{K}%
_{1},J_{3}\right)$ and a Hilbert $\mathcal{B}$-module $\mathcal{K}_{2}$, a
representation $\pi _{\varphi }\ $of $\mathcal{A}$ on $\left( \mathcal{K}%
_{1},J_{1}\right)$, a $\pi_{\varphi }$-representation $\pi_{\mathcal{X}}$ of
$\mathcal{X}$ on $\left(\mathcal{K}_{1},J_{3}\right)$ and $\left(\mathcal{K}%
_{2},J_{4}=\mathrm{id}_{\mathcal{K}_2}\right)$, two operators $V_{\Phi }:%
\mathcal{H}_{1}\to \mathcal{K}_{1}$ and $W_{\Phi }:\mathcal{H}%
_{2}\to \mathcal{K}_{2}$ such that

\begin{enumerate}
\item $V_{\Phi }^{\#}=V_{\Phi }^*$, $\pi _{\varphi }(\alpha
(a))V_{\Phi }=J_{3}\pi _{\varphi }(a)V_{\Phi }J_{1}$ for all $a\in %
\mathcal{A}$, and $W_{\Phi }$ is a coisometry with $W_{\Phi }^{\#}=W_{\Phi
}^*$;

\item $\varphi (a)=V_{\Phi }^{\#}\pi _{\varphi }(a)V_{\Phi }$ for all $a\in %
\mathcal{A}$;

\item $\Phi (x)=W_{\Phi }^{\#}\pi _{\mathcal{X}}(x)V_{\Phi }$ for all $x\in %
\mathcal{X}$.
\end{enumerate}
\end{thm}

\begin{proof}
It is straightforward to observe via condition (ii) in Definition \ref{def4}
that the quotient $\mathcal{A}\otimes _{\mathrm{alg}}\mathcal{H}%
_{1}/N_{\varphi }$ of the algebraic tensor product $\mathcal{A}\otimes _{%
\mathrm{alg}}\mathcal{H}_{1}$, where
\begin{align*}
N_{\varphi }& =\left\{ \sum_{i=1}^{n}a_{i}\otimes \xi
_{i}:\sum_{i=1}^{n}\sum_{j=1}^{n}\left\langle \xi
_{i},\varphi \left( \alpha \left( a_{i}\right)^*a_{j}\right) \xi
_{j}\right\rangle =0\right\} , \\
(a\otimes \xi +N_{\varphi })b& =a\otimes (\xi b)+N_{\varphi },
\end{align*}%
is a pre-Hilbert $\mathcal{B}$-module with the inner product given by
\[
\left\langle \sum_{i=1}^{n}a_{i}\otimes \xi _{i}+N_{\varphi
},\sum_{j=1}^{m}b_{j}\otimes \eta _{j}+N_{\varphi }\right\rangle
=\sum_{i=1}^{n}\sum_{j=1}^{m}\left\langle \xi _{i},\varphi
\left( \alpha \left( a_{i}\right)^*b_{j}\right) \eta
_{j}\right\rangle . 
\]
Let $\mathcal{K}_{1}$ be the Hilbert $\mathcal{B}$-module obtained by the
completion of $\mathcal{A}\otimes _{\mathrm{alg}}\mathcal{H}_{1}/N_{\varphi
} $. Using \ref{def4} (i), it is easy to check that $\left( \mathcal{K}%
_{1},J_{3}\right) $ is a Krein $\mathcal{B}$-module, where the fundamental
symmetry $J_{3}$ is defined by
\[
J_{3}\left( a\otimes \xi +N_{\varphi }\right) =\alpha \left( a\right)
\otimes (J_{1}\xi )+N_{\varphi }. 
\]
Also it is easy to check that the map $V_{\Phi }:\mathcal{H}_{1}\to %
\mathcal{K}_{1}$ defined by
\[
V_{\Phi }\xi =1\otimes J_{1}\xi +N_{\varphi } 
\]
is an operator, and $V_{\Phi }^*(a\otimes \xi +N_{\varphi
})=J_{1}\varphi (a)\xi $. Moreover, for any $\xi \in \mathcal{H}_{1}$, $%
J_{3}V_{\Phi }J_{1}(\xi )=J_{3}(1\otimes \xi +N_{\varphi })=1\otimes
(J_{1}\xi )+N_{\varphi }=V_{\Phi }(\xi )$, which implies that $J_{3}V_{\Phi
}J_{1}=V_{\Phi }$, and so $V_{\Phi }^{\#}=V_{\Phi }^*.$ Also the map $%
\pi _{\varphi }:\mathcal{A}\to \mathcal{L}(\mathcal{K}_{1})$ given
by
\begin{equation}
\pi _{\varphi }(a)\left( b\otimes \xi +N_{\varphi }\right) =ab\otimes \xi
+N_{\varphi }\qquad (a,b\in \mathcal{A},\xi \in \mathcal{H}_{1}) \label{m2}
\end{equation}%
is a representation of $\mathcal{A}$ on $\left( \mathcal{K}_{1},J_{3}\right)
$, in this case, we have $\pi _{\varphi }(a)^*=\pi _{\varphi }(\alpha
(a)^*)$ (for details see \cite[Theorem 4.4]{HHJ}). Since $V_{\Phi
}^{\#}=V_{\Phi }^*$, we have
\begin{equation*}
V_{\Phi }^{\#}\pi _{\varphi }(a)V_{\Phi }\xi =V_{\Phi }^*\left(
a\otimes J_{1}\left( \xi \right) +N_{\varphi }\right) =J_{1}\varphi
(a)J_{1}\xi =\varphi (a)\xi ,
\end{equation*}%
for all $\xi \in \mathcal{H}_{1}$, and so
\begin{equation*}
\varphi (a)=V_{\Phi }^{\#}\pi _{\varphi }(a)V_{\Phi }\qquad (a\in \mathcal{A}%
).
\end{equation*}%
Moreover, $\mathcal{K}_{1}=[\pi _{\varphi }\left( \mathcal{A}\right) V_{\Phi
}\mathcal{H}_{1}]$. For any $a\in \mathcal{A}$ and $\xi \in \mathcal{H}_{1}$%
, we get
\begin{equation*}
\pi _{\varphi }(\alpha \left( a\right) )V_{\Phi }\xi =\alpha \left( a\right)
\otimes J_{1}\left( \xi \right) +N_{\varphi }
\end{equation*}%
and
\begin{equation*}
J_{3}\pi _{\varphi }(a)V_{\Phi }J_{1}\xi =J_{3}\left( a\otimes \xi
+N_{\varphi }\right) =\alpha \left( a\right) \otimes J_{1}\left( \xi \right)
+N_{\varphi }
\end{equation*}%
whence
\[
\pi _{\varphi }(\alpha \left( a\right) )V_{\Phi }=J_{3}\pi _{\varphi
}(a)V_{\Phi }J_{1}\qquad (a\in \mathcal{A}). 
\]
Let $\mathcal{K}_{2}=\left[ \Phi \left( \mathcal{X}\right) \mathcal{H}_{1}%
\right] \subseteq \mathcal{H}_{2}$ be the closed linear subspace. Then $%
\mathcal{K}_{2}$ is a Hilbert $\mathcal{B}$-module. For the inclusion $%
\mathbf{J}_{\mathcal{K}_{2}}$ of $\mathcal{K}_{2}$ into $\mathcal{H}_{2}$,
put $W_{\Phi }^*=\mathbf{J}_{\mathcal{K}_{2}}$ and then $W_{\Phi }$ is
a coisometry and $W_{\Phi }^{\#}=W_{\Phi }^*$. On the other hand, for
any $a,b\in \mathcal{A}$ and $x\in \mathcal{X}$, since $J_{1}V_{\Phi
}^{\#}=J_{1}V_{\Phi }^*=V_{\Phi }^*J_{3}=V_{\Phi }^{\#}J_{3}$,
we obtain
\begin{align*}
\Phi \left( xa\right) ^{\#}\Phi \left( xb\right) & =\varphi \left(
\left\langle xa,xb\right\rangle \right) =\varphi \left( a^*\left\langle x,x\right\rangle b\right) \\
& =V_{\Phi }^{\#}\pi _{\varphi }\left( a^*\right) \pi _{\varphi
}\left( \left\langle x,x\right\rangle b\right) V_{\Phi } \\
& =V_{\Phi }^{\#}J_{3}\pi _{\varphi }\left( a\right)^*J_{3}\pi
_{\varphi }\left( \left\langle x,x\right\rangle b\right) V_{\Phi } \\
& =J_{1}V_{\Phi }^{\#}\pi _{\varphi }\left( a\right)^*J_{3}\pi
_{\varphi }\left( \left\langle x,x\right\rangle b\right) V_{\Phi }.
\end{align*}%
Therefore, for any $a_{1},\cdots ,a_{n}\in \mathcal{A}$ and $\xi _{1},\cdots
,\xi _{n}\in \mathcal{H}_{1}$, we have
\begin{align*}
\left\Vert \sum_{i=1}^{n}\Phi \left( xa_{i}\right) \xi
_{i}\right\Vert ^{2}& =\left\Vert
\sum_{i=1}^{n}\sum_{j=1}^{n}\left\langle \xi _{i},J_{1}\Phi
\left( xa_{i}\right) ^{\#}\Phi \left( xa_{j}\right) \xi _{j}\right\rangle
\right\Vert \\
& =\left\Vert \sum_{i=1}^{n}\sum_{j=1}^{n}\left\langle \pi
_{\varphi }\left( a_{i}\right) V_{\Phi }\xi _{i},J_{3}\pi _{\varphi }\left(
\left\langle x,x\right\rangle \right) \pi _{\varphi }\left( a_{j}\right)
V_{\Phi }\xi _{j}\right\rangle \right\Vert \\
& =\left\Vert \left\langle \sum_{i=1}^{n}\pi _{\varphi }\left(
a_{i}\right) V_{\Phi }\xi _{i},J_{3}\pi _{\varphi }\left( \left\langle
x,x\right\rangle \right) \sum_{j=1}^{n}\pi _{\varphi }\left(
a_{j}\right) V_{\Phi }\xi _{j}\right\rangle \right\Vert \\
& \leq \left\Vert \sum_{i=1}^{n}\pi _{\varphi }\left( a_{i}\right)
V_{\Phi }\xi _{i}\right\Vert ^{2}\left\Vert J_{3}\pi _{\varphi }\left(
\left\langle x,x\right\rangle \right) \right\Vert ,
\end{align*}%
which implies that for each $x\in \mathcal{X}$, there exists a map $\pi _{%
\mathcal{X}}\left( x\right) :\mathcal{K}_{1}\to \mathcal{K}_{2}$
such that
\begin{equation}
\pi _{\mathcal{X}}\left( x\right) \left( \sum_{i=1}^{n}\pi _{\varphi
}\left( a_{i}\right) V_{\Phi }\xi _{i}\right) =\sum_{i=1}^{n}\Phi
\left( xa_{i}\right) \xi _{i}. \label{m5}
\end{equation}%
Then we obtain that
\begin{align*}
&\hspace{-1cm}
\left\langle \pi _{\mathcal{X}}\left( x\right) \left(
\sum_{i=1}^{n}{}\pi _{\varphi }\left( a_{i}\right) V_{\Phi }\xi
_{i}\right) ,\sum_{j=1}^{m}\Phi \left( y_{j}\right) \eta
_{j}\right\rangle\\
&=\left\langle \sum_{i=1}^{n}\Phi \left(
xa_{i}\right) \xi _{i},\sum_{j=1}^{m}\Phi \left( y_{j}\right) \eta
_{j}\right\rangle \\
&=\sum_{j=1}^{m}\sum_{i=1}^{n}\left\langle \xi
_{i},J_{1}\Phi \left( xa_{i}\right) ^{\#}\Phi \left( y_{j}\right) \eta
_{j}\right\rangle \\
&=\sum_{j=1}^{m}\sum_{i=1}^{n}\left\langle \xi _{i},V_{\Phi
}^{\#}\pi _{\varphi }\left( a_{i}\right)^*J_{3}\pi _{\varphi }\left(
\left\langle x,y_{j}\right\rangle \right) V_{\Phi }\eta _{j}\right\rangle \\
&=\left\langle \sum_{i=1}^{n}\pi _{\varphi }\left( a_{i}\right)
V_{\Phi }\xi _{i},\sum_{j=1}^{m}J_{3}\pi _{\varphi }\left(
\left\langle x,y_{j}\right\rangle \right) V_{\Phi }\eta _{j}\right\rangle ,
\end{align*}%
which implies that
\begin{equation*}
\pi _{\mathcal{X}}\left( x\right)^*\left( \sum_{j=1}^{m}\Phi
\left( y_{j}\right) \eta _{j}\right) =\sum_{j=1}^{m}J_{3}\pi
_{\varphi }\left( \left\langle x,y_{j}\right\rangle \right) V_{\Phi }\eta
_{j}.
\end{equation*}%
Therefore, $\pi _{\mathcal{X}}\left( x\right) \in \mathcal{L}\left( %
\mathcal{K}_{1},\mathcal{K}_{2}\right) $. In this way we have obtained a map
$\pi _{\mathcal{X}}:\mathcal{X}\to \mathcal{L}\left( \mathcal{K}_{1},%
\mathcal{K}_{2}\right) $, and then for any $a_{1},\cdots ,a_{n}\in %
\mathcal{A}$ $\xi _{1},\cdots ,\xi _{n}\in \mathcal{H}_{1}$ and $x,y\in %
\mathcal{X}$, we get
\begin{align*}
\pi _{\mathcal{X}}\left( x\right) ^{\#}\pi _{\mathcal{X}}\left( y\right)
\left( \sum_{i=1}^{n}{}\pi _{\varphi }\left( a_{i}\right) V_{\Phi
}\xi _{i}\right) & =J_{3}\pi _{\mathcal{X}}\left( x\right)^*\left(
\sum_{i=1}^{n}\Phi \left( ya_{i}\right) \xi _{i}\right) \\
& =J_{3}\left( \sum_{i=1}^{n}J_{3}\pi _{\varphi }\left( \left\langle
x,y\right\rangle a_{i}\right) V_{\Phi }\xi _{i}\right) \\
& =\pi _{\varphi }\left( \left\langle x,y\right\rangle \right) \left(
\sum_{i=1}^{n}\pi _{\varphi }\left( a_{i}\right) V_{\Phi }\xi
_{i}\right) ,
\end{align*}%
which implies that for any $x,y\in \mathcal{X}$,
\begin{equation}
\pi _{\mathcal{X}}\left( x\right) ^{\#}\pi _{\mathcal{X}}\left( y\right)
=\pi _{\varphi }\left( \left\langle x,y\right\rangle \right) .
\label{eqn:pi-X and pi-varphi}
\end{equation}%
Therefore $\pi _{\mathcal{X}}$ is a $\pi _{\varphi }$-representation of $%
\mathcal{X}$ on the Krein spaces $\left( \mathcal{K}_{1},J_{3}\right) $ and $%
\left( \mathcal{K}_{2},J_{4}\right) $. Moreover,
\begin{equation*}
W_{\Phi }^{\#}\pi _{\mathcal{X}}\left( x\right) V_{\Phi }\xi =W_{\Phi
}^*\Phi \left( x\right) \xi =\Phi \left( x\right) \xi
\end{equation*}%
for all $\xi \in \mathcal{H}_{1}$.
\end{proof}

\begin{rem}
\upshape
In Theorem \ref{main1}, if $\varphi$ is unital, i.e., $\varphi(1)=1$, then
since $V_\Phi^*(a\otimes \xi+N_\varphi)=J_1\varphi(a)\xi$ for any $a\in%
\mathcal{A}$ and $\xi\in\mathcal{H}_1$, $V_\Phi^*
V_\Phi\xi=V_\Phi^*(1\otimes J_1\xi+N_\varphi)=J_1\varphi(1)J_1\xi=\xi$,
which implies that $V_\Phi$ is isometry.
\end{rem}

\noindent A six-tuple $\left( \pi _{\mathcal{X}},\pi _{\varphi },V_{\Phi },W_{\Phi
},\left( \mathcal{K}_{1},J_{3}\right) ,\left( \mathcal{K}_{2},J_{4}\right)
\right) $, which verifying the relations (1)-(3) in Theorem \ref{main1} is
called the KSGNS construction associated to the $\varphi $-map $\Phi $. If $%
\mathcal{K}_{2}=[\pi _{\mathcal{X}}\left( \mathcal{X}\right) V_{\varphi }%
\mathcal{H}_{1}]$ and $\mathcal{K}_{1}=\left[ \pi _{\varphi }\left( %
\mathcal{A}\right) V_{\Phi }\mathcal{H}_{1}\right] $, we say that six-tuple $\left(
\pi _{\mathcal{X}},\pi _{\varphi },V_{\Phi },W_{\Phi },\left( \mathcal{K}%
_{1},J_{3}\right) ,\left( \mathcal{K}_{2},J_{4}\right) \right) $ is \textit{%
minimal}.

\begin{rem}
\upshape The KSGNS construction associated to the $\varphi $-map $\Phi $
constructed in Theorem \ref{main1} is minimal. Using \eqref{m5} we have
\begin{equation*}
\lbrack \pi _{\mathcal{X}}\left( \mathcal{X}\right) V_{\Phi }\mathcal{H}%
_{1}]=[\pi _{\mathcal{X}}\left( \mathcal{X}\right) \pi _{\varphi }\left(
1\right) V_{\Phi }\mathcal{H}_{1}]=\left[ \Phi \left( \mathcal{X}\alpha
\left( 1\right) \right) \mathcal{H}_{1}\right] =\left[ \Phi \left( %
\mathcal{X}\right) \mathcal{H}_{1}\right] =\mathcal{K}_{2}
\end{equation*}%
and by applying \eqref{m2} we get
\begin{equation*}
\left[ \pi _{\varphi }\left( \mathcal{A}\right) V_{\Phi }\mathcal{H}_{1}%
\right] =\mathcal{K}_{1}.
\end{equation*}
\end{rem}


\begin{rem}
\upshape If $\left( \pi _{\mathcal{X}},\pi _{\varphi },V_{\Phi },W_{\Phi
},\left( \mathcal{K}_{1},J_{3}\right) ,\left( \mathcal{K}_{2},J_{4}\right)
\right) $ is a minimal KSGNS construction associated to the $\varphi $-map $%
\Phi $, then $\mathcal{K}_{1}=\left[ \pi _{\varphi }\left( \mathcal{A}%
\right)^*V_{\Phi }\mathcal{H}_{1}\right] $. Indeed, we obtain that
\begin{align*}
\left[ \pi _{\varphi }\left( \mathcal{A}\right)^*V_{\Phi }\mathcal{H}%
_{1}\right] & =\left[ J_{3}\pi _{\varphi }\left( \mathcal{A}^*\right)
J_{3}V_{\Phi }J_{1}\mathcal{H}_{1}\right] =\left[ J_{3}\pi _{\varphi }\left( %
\mathcal{A}\right) V_{\Phi }\mathcal{H}_{1}\right] \\
& =J_{3}\left[ \pi _{\varphi }\left( \mathcal{A}\right) V_{\Phi }\mathcal{H}%
_{1}\right] =J_{3}\mathcal{K}_{1}=\mathcal{K}_{1}.
\end{align*}
\end{rem}


The following result may be considered as a generalization in the context of maps on Krein $C^*$-modules of \cite[Theorem 4.6]{HHJ}.

\begin{prop}
\label{prop1} Let $\mathcal{X}$ be a Hilbert $\mathcal{A}$-module and $\Phi :%
\mathcal{X}\to \mathcal{L}(\mathcal{H}_{1},\mathcal{H}_{2})$ be a $%
\varphi $-map. If the constructions $\left( \pi _{\mathcal{X}},\pi _{\varphi },V_{\Phi
},W_{\Phi },\left( \mathcal{K}_{1},J_{3}\right) ,\left( \mathcal{K}%
_{2},J_{4}\right) \right) $ and $\left( \pi _{\mathcal{X}}^{\prime },\pi
_{\varphi }^{\prime },V_{\Phi }^{\prime },W_{\Phi }^{\prime },\left( %
\mathcal{K}_{1}^{\prime },J_{3}^{\prime }\right) ,\left( \mathcal{K}%
_{2}^{\prime },J_{4}^{\prime }\right) \right) $ are two minimal KSGNS
construction for $\Phi $, then there are two unitary operators $U_{1}:%
\mathcal{K}_{1}\to \mathcal{K}_{1}^{\prime }$ and $U_{2}:\mathcal{K}%
_{2}\to \mathcal{K}_{2}^{\prime }$ such that

\begin{enumerate}
\item $U_{1}V_{\Phi}=V_{\Phi}^{\prime}$ and $U_{1}\pi_{\varphi
}\left(a\right) U_{1}^{\#}=\pi_{\varphi }^{\prime}\left(a\right)$ for all $%
a\in\mathcal{A}$;

\item $U_{2}W_{\Phi}=W_{\Phi}^{\prime}$ and $U_{2}\pi_{\mathcal{X}}\left(
x\right)U_{1}^{\#}=\pi_{\mathcal{X}}^{\prime}\left(x\right)$ for all $x\in %
\mathcal{X}$.
\end{enumerate}
\end{prop}

\begin{proof}
From Theorem (1) \ref{main1}, we have
\begin{align*}
&\hspace{-2.5cm}\left\langle \sum_{i=1}^{n}\pi _{\varphi }^{\prime }\left(
a_{i}\right) V_{\Phi }^{\prime }\xi _{i},\sum_{j=1}^{m}\pi _{\varphi
}^{\prime }\left( b_{j}\right) V_{\Phi }^{\prime }\eta _{j}\right\rangle\\
&
=\left\langle \sum_{i=1}^{n}V_{\Phi }^{\prime }\xi
_{i},\sum_{j=1}^{m}V_{\Phi }^{\prime \ast }\pi _{\varphi }^{\prime
}\left( a_{i}\right)^*\pi _{\varphi }^{\prime }\left( b_{j}\right)
V_{\Phi }^{\prime }\eta _{j}\right\rangle \\
& =\sum_{i=1}^{n}\sum_{j=1}^{m}\left\langle \xi _{i},\left(
V_{\Phi }^{\prime }\right) ^{\#}J_{3}\pi _{\varphi }^{\prime }\left(
a_{i}^*\right) J_{3}\pi _{\varphi }^{\prime }\left( b_{j}\right)
V_{\Phi }^{\prime }\eta _{j}\right\rangle \\
& =\sum_{i=1}^{n}\sum_{j=1}^{m}\left\langle \xi _{i},\left(
V_{\Phi }^{\prime }\right) ^{\#}J_{3}\pi _{\varphi }^{\prime }\left(
a_{i}^*\right) \pi _{\varphi }^{\prime }\left( \alpha \left(
b_{j}\right) \right) V_{\Phi }^{\prime }J_{1}\eta _{j}\right\rangle \\
& =\sum_{i=1}^{n}\sum_{j=1}^{m}\left\langle \xi _{i},\left(
V_{\Phi }^{\prime }\right) ^{\#}\pi _{\varphi }^{\prime }\left( \alpha
\left( a_{i}^*\right) b_{j}\right) V_{\Phi }^{\prime }\eta
_{j}\right\rangle \\
& =\sum_{i=1}^{n}\sum_{j=1}^{m}\left\langle \xi _{i},\varphi
\left( a_{i}^{\#}b_{j}\right) \eta _{j}\right\rangle \\
& =\left\langle \sum_{i=1}^{n}\pi _{\varphi }\left( a_{i}\right)
V_{\Phi }\xi _{i},\sum_{j=1}^{m}\pi _{\varphi }\left( b_{j}\right)
V_{\Phi }\eta _{j}\right\rangle ,
\end{align*}%
which implies that there is a unitary operator $U_{1}:\mathcal{K}%
_{1}\to \mathcal{K}_{1}^{\prime }$ such that
\begin{equation*}
U_{1}\left( \sum_{i=1}^{n}\pi _{\varphi }\left( a_{i}\right) V_{\Phi
}\xi _{i}\right) =\sum_{i=1}^{n}\pi _{\varphi }^{\prime }\left(
a_{i}\right) V_{\Phi }^{\prime }\xi _{i}.
\end{equation*}
Then it is easy to check that $U_{1}V_{\Phi }=V_{\Phi }^{\prime }$ and $%
U_{1}\pi _{\varphi }\left( a\right) =\pi _{\varphi }^{\prime }\left(
a\right) U_{1}$ for all $a\in \mathcal{A}$. Also, for each $a_{1}, \cdots , a_{n}, b_{1}, \cdots , b_{m} \in \mathcal{A}$,
$\xi _{1}, \cdots , \xi _{n}, \eta _{1},\cdots , \eta _{m} \in \mathcal{H}_{1}$ and each $a \in \mathcal{A}$
we obtain that
\begin{align*}
U_{1}\pi _{\varphi }\left( a\right) U_{1}^{\#}\left(
\sum_{i=1}^{n}\pi _{\varphi }^{\prime }\left( a_{i}\right) V_{\Phi
}^{\prime }\xi _{i}\right) & =U_{1}\pi _{\varphi }\left( a\right)
J_{3}U_{1}^*J_{3}^{\prime }\left( \sum_{i=1}^{n}\pi _{\varphi
}^{\prime }\left( a_{i}\right) V_{\Phi }^{\prime }\xi _{i}\right) \\
& =U_{1}\pi _{\varphi }\left( a\right) J_{3}U_{1}^*\left(
\sum_{i=1}^{n}\pi _{\varphi }^{\prime }\left( \alpha \left(
a_{i}\right) \right) V_{\Phi }^{\prime }J_{1}\xi _{i}\right) \\
& =U_{1}\pi _{\varphi }\left( a\right) J_{3}\left( \sum_{i=1}^{n}\pi
_{\varphi }\left( \alpha \left( a_{i}\right) \right) V_{\Phi }^{\prime
}J_{1}\xi _{i}\right) \\
& =U_{1}\pi _{\varphi }\left( a\right) \left( \sum_{i=1}^{n}\pi
_{\varphi }\left( a_{i}\right) V_{\Phi }\xi _{i}\right)\\
&=\sum_{i=1}^{n}\pi _{\varphi }^{\prime }\left( aa_{i}\right) V_{\Phi
}^{\prime }\xi _{i} \\
& =\pi _{\varphi }^{\prime }\left( a\right) \left( \sum_{i=1}^{n}\pi
_{\varphi }^{\prime }\left( a_{i}\right) V_{\Phi }^{\prime }\xi _{i}\right) ,
\end{align*}
and so $U_{1}\pi _{\varphi }\left( a\right) U_{1}^{\#}=\pi _{\varphi
}^{\prime }\left( a\right) $ for all $a\in \mathcal{A}$. Also, from (3) of
in Theorem \ref{main1} we observe that
\begin{align*}
&\hspace{-2cm}\left\langle \sum_{i=1}^{n}\pi _{\mathcal{X}}^{\prime }\left(
x_{i}\right) V_{\Phi }^{\prime }\xi _{i},\sum_{j=1}^{m}\pi _{%
\mathcal{X}}^{\prime }\left( y_{j}\right) V_{\Phi }^{\prime }\eta
_{j}\right\rangle\\& =\left\langle \sum_{i=1}^{n}\left( W_{\Phi
}^{^{\prime }}\right) ^{\#}\pi _{\mathcal{X}}^{\prime }\left( x_{i}\right)
V_{\Phi }^{\prime }\xi _{i},\sum_{j=1}^{m}\left( W_{\Phi }^{^{\prime
}}\right) ^{\#}\pi _{\mathcal{X}}^{\prime }\left( y_{j}\right) V_{\Phi
}^{\prime }\eta _{j}\right\rangle \\
& =\sum_{i=1}^{n}\sum_{j=1}^{m}\left\langle \Phi \left(
x_{i}\right) \xi _{i},\Phi \left( y_{j}\right) \eta _{j}\right\rangle \\
& =\left\langle \sum_{i=1}^{n}W_{\Phi }^{\#}\pi _{\mathcal{X}}\left(
x_{i}\right) V_{\Phi }\xi _{i},\sum_{j=1}^{m}W_{\Phi }^{\#}\pi _{%
\mathcal{X}}\left( y_{j}\right) V_{\Phi }\eta _{j}\right\rangle \\
& =\left\langle \sum_{i=1}^{n}\pi _{\mathcal{X}}\left( x_{i}\right)
V_{\Phi }\xi _{i},\sum_{j=1}^{m}\pi _{\mathcal{X}}\left(
y_{j}\right) V_{\Phi }\eta _{j}\right\rangle ,
\end{align*}%
which implies that there is a unitary operator $U_{2}:\mathcal{K}%
_{2}\to \mathcal{K}_{2}^{\prime }$ such that
\begin{equation*}
U_{2}\left( \sum_{i=1}^{n}\pi _{\mathcal{X}}\left( x_{i}\right)
V_{\Phi }\xi _{i}\right) =\sum_{i=1}^{n}\pi _{\mathcal{X}}^{\prime
}\left( x_{i}\right) V_{\Phi }^{\prime }\xi _{i}.
\end{equation*}%
Then by using (3) of Theorem \ref{main1} we obtain that
\begin{align*}
W_{\Phi }^*U_{2}^*\left( \sum_{i=1}^{n}\pi _{\mathcal{X}%
}^{\prime }\left( x_{i}\right) V_{\Phi }^{\prime }\xi _{i}\right) &
=\sum_{i=1}^{n}W_{\Phi }^*\pi _{\mathcal{X}}\left(
x_{i}\right) V_{\Phi }\xi _{i}=\sum_{i=1}^{n}W_{\Phi }^{\#}\pi _{%
\mathcal{X}}\left( x_{i}\right) V_{\Phi }\xi _{i} \\
& =\sum_{i=1}^{n}\Phi \left( x_{i}\right) \xi
_{i}=\sum_{i=1}^{n}\left( W_{\Phi }^{\prime }\right) ^{\#}\pi _{%
\mathcal{X}}^{\prime }\left( x_{i}\right) V_{\Phi }^{\prime }\xi _{i} \\
& =\left( W_{\Phi }^{\prime }\right)^*\sum_{i=1}^{n}\pi _{%
\mathcal{X}}^{\prime }\left( x_{i}\right) V_{\Phi }^{\prime }\xi _{i}
\end{align*}%
and, since $\mathcal{K}_{2}^{\prime }=[\pi _{\mathcal{X}}^{\prime }\left( %
\mathcal{X}\right) ^{\prime }V_{\Phi }^{\prime }\mathcal{H}_{1}]$, and so $%
U_{2}W_{\Phi }=W_{\Phi }^{\prime }$. On the other hand, by applying (1) of
Theorem \ref{main1} and \eqref{eq2.3} for any $a_{1},\cdots ,a_{n}\in %
\mathcal{A}$, $\xi _{1},\cdots ,\xi _{n}\in \mathcal{H}_{1}$ and $x\in %
\mathcal{X}$, we obtain that
\begin{align*}
U_{2}\pi _{\mathcal{X}}\left( x\right) U_{1}^{\#}\left(
\sum_{i=1}^{n}\pi _{\varphi }^{\prime }\left( a_{i}\right) V_{\Phi
}^{\prime }\xi _{i}\right) & =U_{2}\pi _{\mathcal{X}}\left( x\right)
J_{3}U_{1}^*J_{3}^{\prime }\left( \sum_{i=1}^{n}\pi _{\varphi
}^{\prime }\left( a_{i}\right) V_{\Phi }^{\prime }\xi _{i}\right) \\
& =U_{2}\pi _{\mathcal{X}}\left( x\right) \left( \sum_{i=1}^{n}\pi
_{\varphi }\left( a_{i}\right) V_{\Phi }\xi _{i}\right) \\
& =U_{2}\left( \sum_{i=1}^{n}\pi _{\mathcal{X}}\left( xa_{i}\right)
V_{\Phi }\xi _{i}\right) \\
&=\sum_{i=1}^{n}\pi _{\mathcal{X}}^{\prime
}\left( xa_{i}\right) V_{\Phi }\xi _{i} \\
& =\pi _{\mathcal{X}}^{\prime }\left( x\right) \left(
\sum_{i=1}^{n}\pi _{\varphi }^{\prime }\left( a_{i}\right) V_{\Phi
}^{\prime }\xi _{i}\right)
\end{align*}%
and taking into account that $\mathcal{K}_{1}=\left[ \pi _{\varphi }^{\prime
}\left( \mathcal{A}\right) V_{\Phi }^{\prime }\mathcal{H}_{1}\right] $, we
deduce that $\pi _{\mathcal{X}}^{\prime }\left( x\right) =U_{2}\pi _{%
\mathcal{X}}\left( x\right) U_{1}^{\#}$ for all $x\in \mathcal{X}$.
\end{proof}


\section{Covariant $\protect\alpha$-CP maps}

Let $\mathcal{G}$ be a locally compact group and let $\mathcal{X}$ be a full
Hilbert $C^*$-module over a unital $C^*$-algebra $\mathcal{A}$. An
action of $\mathcal{G}$ on $\mathcal{X}$ is a group morphism $\eta $ from $%
\mathcal{G}$ to Aut$\left( \mathcal{X}\right) $, the group of all Hilbert $%
C^*$-module isomorphisms from $\mathcal{X}$ onto $\mathcal{X}$, such
that the map $t\mapsto \eta_{t}\left( x\right) $ from $\mathcal{G}$ to $%
\mathcal{X}$ is continuous for each $x\in \mathcal{X}$. The triple $(%
\mathcal{G},\eta ,\mathcal{X})$ is called a dynamical system on Hilbert $%
C^*$-modules (see,\cite{K, J}).

An action $t\mapsto \eta_{t}$ of $\mathcal{G}$ on $\mathcal{X}$ induces a
unique action $t\mapsto \beta_{t}^{\eta}$ of $\mathcal{G}$ on $\mathcal{A}$
such that $\beta_{t}^{\eta}\left( \left\langle x,y\right\rangle \right)
=\left\langle \eta_{t}\left( x\right) ,\eta _{t}\left( y\right)
\right\rangle $ for all $x,y\in \mathcal{X},t\in \mathcal{G}$; see \cite{K,
J}.

A pseudo-unitary representation of $\mathcal{G}$ on a Krein space $(%
\mathcal{H},J)$ is a map $t\mapsto u_{t}$ from $\mathcal{G}$ to $\mathcal{L}%
\left(\mathcal{H}\right) $ such that $u_{e}=\mathrm{id}_{\mathcal{H}}$, $%
u_{ts}=u_{t}u_{s}$ and $u_{t^{-1}}=u_{t}^{\#}$ for all $s,t\in \mathcal{G}$.
It follows from $u_{t^{-1}}=u_{t}^{\#}$ that $u_{t^{-1}}=Ju_{t}^*J$. So
\[
Ju_{t^{-1}}=u_{t}^*J.
\]

When we deal with usual Hilbert spaces $u_{t}^{\#}$ is replaced by $%
u_{t}^*$, and the pseudo-unitary representation is replaced by unitary
representation.

Let $t\mapsto u_{t}$ and $t\mapsto $ $u_{t}^{\prime}$ be two pseudo-unitary $%
\ast $-representations of $\mathcal{G}$ on Krein spaces $(\mathcal{H}%
_{1},J_{1})$ and $(\mathcal{H}_{2},J_{2})$.

\begin{defn}
\upshape
A $\varphi $-map $\Phi :\mathcal{X}\to \mathcal{L}(\mathcal{H}_{1},%
\mathcal{H}_{2})$ is said to be $\left(u^{\prime},u\right)$-covariant with
respect to $\eta$ if
\[
\Phi \left( \eta_{t}\left( x\right) \right) =u_{t}^{\prime}\Phi \left(
x\right) u_{t}^{\#}
\]
for all $t\in \mathcal{G}$, $x\in \mathcal{X}$, and
\[
\beta_{t}^{\eta}\circ \alpha =\alpha \circ \beta_{t}^{\eta}
\]
for all $t\in \mathcal{G}$.
\end{defn}


\begin{rem}
\upshape
Let $\Phi :\mathcal{X}\to \mathcal{L}(\mathcal{H}_{1},\mathcal{H}%
_{2})$ be a $\varphi$-map. If $\Phi $ is $\left(u^{\prime},u\right)$%
-covariant with respect to $\eta$, then $\varphi$ is $u$-covariant with
respect to $\beta^{\eta}$, which means that
\[
\varphi (\beta_{t}^{\eta}(a)) =u_{t}\varphi (a)u_{t}^*\qquad (a\in %
\mathcal{A},t\in \mathcal{G}).
\]
In fact, for any $x,y\in \mathcal{X}$ and $t\in \mathcal{G}$, we obtain that
\begin{align}
\varphi\left(\beta_{t}^{\eta}\left(\left\langle
x,y\right\rangle\right)\right) &=\varphi\left(\left\langle
\eta_{t}\left(x\right), \eta_{t}\left( y\right) \right\rangle \right)
=\Phi\left(\eta_{t}\left(x\right)\right)^{\#}\Phi\left(\eta_{t}\left(y%
\right)\right) \\
&=u_{t}\Phi\left(x\right)^{\#}\left(u_{t}^{\prime}\right)^{\#}u_{t}^{\prime}%
\Phi\left( y\right) u_{t}^{\#} \\
&=u_{t}\varphi \left(\left\langle x,y\right\rangle\right) u_{t}^{\#}.
\label{eqn:covariant of varphi}
\end{align}
\end{rem}

Let $t\mapsto u_{t}$ and $t\mapsto $ $u_{t}^{\prime}$ be two unitary
representations of $\mathcal{G}$ on Krein spaces $\left( \mathcal{H}%
_{1},J_{1}\right) $ and $\left( \mathcal{H}_{2},J_{2}\right) $.

A representation $\pi_{\mathcal{\mathcal{X}}}$ of $\mathcal{X}$ on Krein
spaces $\left( \mathcal{H}_{1},J_{1}\right) $ and $\left( \mathcal{H}%
_{2},J_{2}\right) $ is $\left( u^{\prime},u\right) $-covariant with respect
to $\eta $ if
\begin{equation*}
\pi_{\mathcal{X}}\left( \eta_{t}\left( x\right) \right) =u_{t}^{\prime }\pi_{%
\mathcal{X}}\left( x\right) u_{t}^{\#}
\end{equation*}%
for all $t\in \mathcal{G},x\in \mathcal{X}$.

If $\pi_{\mathcal{X}}$ is $\left( u^{\prime},u\right) $-covariant with
respect to $\eta ,$ $u_{t}^{\#}=u_{t}^*$ for all $t\in \mathcal{G}$ and
$\pi_{\mathcal{A}}$ is continuous, then $\pi_{\mathcal{A}}$ is $u$-covariant
with respect to $\beta ^{\eta}$ in the sense that $\pi_{\mathcal{A}%
}(\beta_{t}^{\eta}(a))=u_{t}\pi_{\mathcal{A}}(a)u_{t}^{\#}$ for all $a\in %
\mathcal{A},t\in \mathcal{G}$. Indeed, by similar arguments as used in %
\eqref{eqn:covariant of varphi}, we have
\begin{align*}
\pi_{\mathcal{A}}\left(\beta_{t}^{\eta}\left(\left\langle
x,y\right\rangle\right)\right) &=u_{t}\pi_{\mathcal{X}}\left( x\right)
^{\#}\pi_{\mathcal{X}}\left(y\right) u_{t}^{\#} \\
&=u_{t}\pi_{\mathcal{A}}\left( \left\langle x,y\right\rangle
\right)u_{t}^{\#}
\end{align*}%
for all $x,y\in \mathcal{X}$ and $t\in \mathcal{G}$, whence, since $%
\mathcal{X}$ is full and $\pi_{\mathcal{A}}$ is continuous, we deduce that $%
\pi_{\mathcal{A}}$ is $u$-covariant with respect to $\beta^{\eta}$.

From now on, we assume that $(\mathcal{H}_1,J_1)$ is a Krein space and $u$
is simultaneously unitary and pseudo-unitary representation of $\mathcal{G}$
on the Krein space $(\mathcal{H}_1,J_1)$, i.e.,
\[
u_t^*u_t=u_tu_t^*=1,\qquad u_t^{\#}u_t=u_tu_t^{\#}=1
\]
for all $t\in\mathcal{G}$, which also implies that $u_t^*=u_t^{\#}$,
equivalently, $J_1u_t=u_tJ_1$ for all $t\in\mathcal{G}$.

\begin{lem}
\label{lmm:unitary and pseudo-unitary v} Let $\mathcal{X}$ be a Hilbert $\mathcal{A}$-module, $\Phi :\mathcal{X}\to
\mathcal{L}(\mathcal{H}_{1},\mathcal{H}_{2})$ be a $\varphi $-map, which is $%
\left(u^{\prime},u\right)$-covariant with respect to $\eta$ and $\left(%
\mathcal{K}_{1},J_{3}\right)$ be the Krein space constructed in Theorem \ref%
{main1}. Then there is a simultaneously unitary and pseudo-unitary
representation $v$ of $\mathcal{G}$ on $\mathcal{K}_1$ such that
\begin{align} \label{eqn:def of v-t}
v_t(a\otimes\xi+N_\varphi)&=\beta_t^{\eta}(a)\otimes u_t(\xi)+N_{\varphi},
\end{align}
\begin{align} \label{eqn:V-u-v-V}
V_{\Phi}u_{t}&=v_{t}V_{\Phi}
\end{align}
for any $a\in\mathcal{A}$, $\xi\in\mathcal{X}$ and $t\in\mathcal{G}$.
\end{lem}

\begin{proof}
For any $t\in\mathcal{G}$, $a\in \mathcal{A}$ and $\xi\in \mathcal{H}_{1}$,
we obtain that
\begin{align*}
\left\|\beta _{t}^{\eta}\left(a\right) \otimes
u_t(\xi)+N_{\varphi}\right\|^2 &=\left\langle \xi,u_{t}^*\varphi
\left(\beta_{t}^{\eta}\left(\alpha(a)^*\right)\beta _{t}^{\eta}
\left(a\right) \right) u_t\left(\xi\right) \right\rangle \\
&=\left\langle \xi,u_t^*u_t\varphi
\left(\alpha(a)^*b\right)u_t^{\#}u_t\left(\xi\right)\right\rangle \\
&=\left\langle a\otimes \xi +N_{\varphi},a\otimes \xi
+N_{\varphi}\right\rangle,
\end{align*}
which implies that there is an isometry operator $v_{t}:\mathcal{K}%
_{1}\to\mathcal{K}_{1}$ such that \eqref{eqn:def of v-t} holds. By
similar arguments, we can easily see that $v_t^*(a\otimes
\xi)=\beta_{-t}^\eta(a)\otimes u_t\xi$ for all $a\in\mathcal{A}$ and $\xi\in%
\mathcal{H}_1$. Therefore, for any $t\in\mathcal{G}$, $v_t$ is
simultaneously unitary and pseudo-unitary. Moreover, we observe that
\begin{align*}
V_{\Phi}u_{t}\left(\xi\right) &=1\otimes J_1u_{t}\left(\xi\right)+N_{\varphi}
=\beta _{t}^{\eta}\left(1\right) \otimes u_{t}\left(J_1\xi\right)
+N_{\varphi} =v_{t}\left(1\otimes J_1\xi +N_{\varphi}\right)\\
&=v_{t}V_{\Phi}\left(\xi \right)
\end{align*}
for all $t\in \mathcal{G},\xi \in \mathcal{H}_{1}$, and so we have %
\eqref{eqn:V-u-v-V}.
\end{proof}

The next result is a variant of \cite[Theorem 3.2]{HJK}.

\begin{thm}
\label{main2} Let $\mathcal{X}$ be a Hilbert $\mathcal{A}$-module, $\Phi :\mathcal{X}\to \mathcal{L}(\mathcal{H}_{1},%
\mathcal{H}_{2})$ be a $\varphi$-map, which is $\left(u^{\prime},u\right)$%
-covariant with respect to $\eta$. Then there are two Krein spaces $\left(%
\mathcal{K}_{1},J_{3}\right)$ and $\left(\mathcal{K}_{2},J_{4}\right)$, two
pseudo-unitary representations $v$ and $v^{\prime }$ of $\mathcal{G}$ on the
Krein spaces $\left(\mathcal{K}_{1},J_{3}\right)$ and $\left(\mathcal{K}%
_{2},J_{4}\right) $, respectively, a representation $\pi_{\varphi}$ of $%
\mathcal{A}$ on $\left(\mathcal{K}_{1},J_{3}\right)$, a $\pi_{\varphi}$%
-representation $\pi_{\mathcal{X}}$ of $\mathcal{X}$ on $\left(\mathcal{K}%
_{1},J_{3}\right) $ and $\left(\mathcal{K}_{2},J_{4}\right) $, which is $%
\left(v^{\prime},v\right)$-covariant and two operators $V_{\Phi }:\mathcal{H}%
_{1}\to \mathcal{K}_{1}$ and $W_{\Phi }:\mathcal{H}_{2}\to %
\mathcal{K}_{2}$ such that

\begin{enumerate}
\item $V_{\Phi }^{\#}=V_{\Phi }^*$, $\pi _{\varphi }(\alpha
(a))V_{\Phi }=J_{3}\pi _{\varphi }(a)V_{\Phi }J_{1}$ for all $a\in %
\mathcal{A}$, and $W_{\Phi }$ is a coisometry with $W_{\Phi }^{\#}=W_{\Phi
}^*$;

\item $\varphi \left( a\right) =V_{\Phi }^{\#}\pi _{\varphi }\left( a\right)
V_{\Phi }$ for all $a\in \mathcal{A}$;

\item $\Phi (x)=W_{\Phi }^{\#}\pi _{\mathcal{X}}(x)V_{\Phi }$ for all $x\in %
\mathcal{X}$ and $a\in \mathcal{A}$;

\item $v_{t}^{\#}=v_{t}^*$ and $\left( v_{t}^{\prime }\right)
^{\#}=\left( v_{t}^{\prime }\right)^*$ for all $t\in \mathcal{G};$

\item $V_{\Phi }u_{t}=v_{t}V_{\Phi }$ and $W_{\Phi }u_{t}^{\prime
}=v_{t}^{\prime }W_{\Phi }$ for all $t\in \mathcal{G}$.
\end{enumerate}
\end{thm}

\begin{proof}
Let $\left( \pi _{\mathcal{X}},\pi _{\varphi },V_{\Phi },W_{\Phi },\left( %
\mathcal{K}_{1},J_{3}\right) ,\left( \mathcal{K}_{2},J_{4}\right) \right) $
be the KSGNS construction associated to the $\varphi $-map $\Phi $
constructed in Theorem \ref{main1}. Then by Lemma \ref{lmm:unitary and
pseudo-unitary v} there is a simultaneously unitary and pseudo-unitary
representation $v$ of $\mathcal{G}$ on $\mathcal{K}_{1}$ satisfying %
\eqref{eqn:def of v-t} and \eqref{eqn:V-u-v-V}. Moreover, $\pi _{\varphi }$
is $v$-covariant with respect to $\beta ^{\eta }$, since $%
v_{t}^{\#}=v_{-t}=v_{t}^*$ and
\begin{align*}
v_{t}\pi _{\varphi }\left( a\right) v_{t}^{\#}\left( b\otimes \xi
+N_{\varphi }\right) & =v_{t}\pi _{\varphi }\left( a\right) \left( \beta
_{t^{-1}}^{\eta }\left( b\right) \otimes u_{t}^*\xi +N_{\varphi
}\right)\\
&=v_{t}\left( a\beta _{t^{-1}}^{\eta }\left( b\right) \otimes
u_{t}^*\xi +N_{\varphi }\right) \\
& =\beta _{t}^{\eta }\left( a\beta _{t^{-1}}^{\eta }\left( b\right) \right)
\otimes u_{t}u_{t}^*\xi +N_{\varphi }\\
&=\beta _{t}^{\eta }\left(
a\right) b\otimes \xi +N_{\varphi } \\
& =\pi _{\varphi }\left( \beta _{t}^{\eta }\left( a\right) \right)
\end{align*}%
for all $a,b\in A$, $\xi \in \mathcal{H}_{1}$ and $t\in \mathcal{G}$. Since $%
\Phi $ is $(u^{\prime },u)$-covariant,
\begin{equation*}
u_{t}^{\prime }\left( \sum_{i=1}^{n}\Phi \left( x_{i}\right) \xi
_{i}\right) =\sum_{i=1}^{n}{}\Phi \left( \eta _{t}\left(
x_{i}\right) \right) u_{t}\xi _{i}
\end{equation*}%
for all $t\in \mathcal{G}$, $x_{i}\in \mathcal{X}$ and $\xi _{i}\in %
\mathcal{H}_{1}\,\,(i=1,\cdots ,n)$, hence $\left[ \Phi (\mathcal{X})%
\mathcal{H}_{1}\right] =\mathcal{K}_{2}$ is invariant under $u_{t}^{\prime }$
as well as $(u_{t}^{\prime })^*$. Therefore, since $W_{\Phi }$ is the
projection onto $\left[ \Phi (\mathcal{X})\mathcal{H}_{1}\right] $, we have $%
u_{t}^{\prime }W_{\Phi }=W_{\Phi }u_{t}^{\prime }$ on $\mathcal{K}_{2}$. Let
$v_{t}^{\prime }=u_{t}^{\prime }|_{\mathcal{K}_{2}}$ for all $t\in \mathcal{G%
}$. Then $t\mapsto v_{t}^{\prime }$ is a unitary representation of $\mathcal{%
G}$ on $\mathcal{K}_{2}$ such that $W_{\Phi }u_{t}^{\prime }=v_{t}^{\prime
}W_{\Phi }$ and since $J_{4}=\mathrm{id}_{\mathcal{K}_{2}}$, $v^{\prime }$
also can be considered as a pseudo-unitary representation of $\mathcal{G}$
on $\left( \mathcal{K}_{2},J_{4}\right) $. On the other hand, for any $t\in
\mathcal{G}$ and $x\in \mathcal{X}$, we obtain that
\begin{align*}
v_{t}^{\#}\left( \sum_{i=1}^{n}\pi _{\varphi }\left( a_{i}\right)
V_{\Phi }\xi _{i}\right) & =v_{t}^*\left(
\sum_{i=1}^{n}a_{i}\otimes J_{1}\xi _{i}\right)
=\sum_{i=1}^{n}\beta _{t^{-1}}^{\eta }(a_{i})\otimes u_{t}^*J_{1}\xi _{i} \\
& =\sum_{i=1}^{n}\pi _{\varphi }\left( \beta _{t^{-1}}^{\eta
}(a_{i})\right) V_{\Phi }u_{t}^*\xi _{i}
\end{align*}%
and, since $\eta _{t}\left( xa\right) =\eta _{t}\left( x\right) \beta
_{t}^{\eta }\left( a\right) $, for any $a_{1},\cdots ,a_{n}\in \mathcal{A}$
and $\xi _{1},\cdots ,\xi _{n}\in \mathcal{H}_{1}$, we have
\begin{align*}
v_{t}^{\prime }\pi _{\mathcal{X}}\left( x\right) v_{t}^{\#}\left(
\sum_{i=1}^{n}\pi _{\varphi }\left( a_{i}\right) V_{\Phi }\xi
_{i}\right) & =v_{t}^{\prime }\pi _{\mathcal{X}}\left( x\right) \left(
\sum_{i=1}^{n}\pi _{\varphi }\left( \beta _{t^{-1}}^{\eta
}(a_{i})\right) V_{\Phi }u_{t}^*\xi _{i}\right) \\
& =v_{t}^{\prime }\left( \sum_{i=1}^{n}\Phi \left( x\beta
_{t^{-1}}^{\eta }(a_{i})\right) u_{t}^*\xi _{i}\right) \\
& =\sum_{i=1}^{n}\Phi \left( \eta _{t}\left( x\beta _{t^{-1}}^{\eta
}(a_{i})\right) \right) u_{t}u_{t}^*\xi
_{i}\\
&=\sum_{i=1}^{n}\Phi \left( \eta _{t}\left( x\right) a_{i}\right)
\xi _{i} \\
& =\pi _{\mathcal{X}}\left( \eta _{t}\left( x\right) \right) \left(
\sum_{i=1}^{n}\pi _{\varphi }\left( a_{i}\right) V_{\Phi }\xi
_{i}\right) ,
\end{align*}%
which implies that $\pi _{\mathcal{X}}$ is $\left( v^{\prime },v\right) $%
-covariant, since $\left[ \pi _{\varphi }\left( \mathcal{A}\right) V_{\Phi }%
\mathcal{H}_{1}\right] =\mathcal{K}_{1}$.
\end{proof}

We say that $(\pi _{\Phi },\pi _{\varphi },V_{\Phi },W_{\Phi },v,v^{\prime
},\left( \mathcal{K}_{1},J_{3}\right) ,\left( \mathcal{K}_{2},J_{4}\right) )$
is a covariant KSGNS construction associated to the $\varphi $-map $\Phi $
being $\left( u^{\prime },u\right) $-covariant with respect to $\eta $. If $%
\mathcal{K}_{2}=[\pi _{\Phi }\left( \mathcal{X}\right) V_{\varphi }%
\mathcal{H}_{1}]$ and $\mathcal{K}_{1}=\left[ \pi _{\varphi }\left( %
\mathcal{A}\right) V_{\Phi }\mathcal{H}_{1}\right] $, we say that $(\pi
_{\Phi },\pi _{\varphi },V_{\Phi },W_{\Phi },v,v^{\prime },$ $\left( %
\mathcal{K}_{1},J_{3}\right) ,\left( \mathcal{K}_{2},J_{4}\right) )$ is
\textit{minimal}.

\noindent The next proposition is a variant of \cite[Theorem 3.5]{HJK}.

\begin{prop}
Suppose that $\mathcal{X}$ is a Hilbert $\mathcal{A}$-module, $\Phi :\mathcal{X}\to \mathcal{L}(\mathcal{H}_{1},\mathcal{H}%
_{2})$ is a $\varphi $-map and $\left( u^{\prime },u\right) $-covariant with
respect to $\eta $. If $(\pi _{\Phi },\pi _{\varphi },V_{\Phi },W_{\Phi
},v,v^{\prime },\left( \mathcal{K}_{1},J_{3}\right) ,\left( \mathcal{K}%
_{2},J_{4}\right) )$ and $(\pi _{\Phi }^{\prime },\pi _{\varphi }^{\prime
},V_{\Phi }^{\prime },W_{\Phi }^{\prime },w,w^{\prime },$ $\left( \mathcal{K}%
_{1}^{\prime },J_{3}^{\prime }\right) ,\left( \mathcal{K}_{2}^{\prime
},J_{4}^{\prime }\right) )$ are two minimal covariant KSGNS constructions
for $\Phi $, then there exist two unitary operators $U_{1}:$ $\mathcal{K}%
_{1}\to \mathcal{K}_{1}^{\prime }$ and $U_{2}:$ $\mathcal{K}%
_{2}\to \mathcal{K}_{2}^{\prime }$ such that

\begin{enumerate}
\item $U_{1}V_{\Phi}=V_{\Phi}^{\prime},$ $U_{1}\pi_{\varphi}\left(a\right)
U_{1}^{\#}=\pi_{\varphi}^{\prime}\left(a\right) $ for all $a\in \mathcal{A} $
and $U_{1}v_{t}U_{1}^{\#}=w_{t}$ for all $t\in \mathcal{G}$

\item $U_{2}W_{\Phi}=W_{\Phi}^{\prime},$ $U_{2}\pi_{\mathcal{X}%
}\left(x\right) U_{1}^{\#}=\pi_{\mathcal{X}}^{\prime}\left(x\right) $ for
all $x\in \mathcal{X}$, and $U_{2}v_{t}^{\prime}U_{2}^{\#}=w_{t}^{\prime}$
for all $t\in \mathcal{G}$.
\end{enumerate}
\end{prop}

\begin{proof}
It is easy to see that the representations $(\pi_{\mathcal{X}},\pi_{\varphi},V_{\Phi},W_{\Phi}, \left(%
\mathcal{K}_{1},J_{3}\right),\left(\mathcal{K}_{2},J_{4}\right))$ and $(\pi_{%
\mathcal{X}}^{\prime},\pi_{\varphi}^{\prime},V_{\Phi}^{\prime},
W_{\Phi}^{\prime},\left(\mathcal{K}_{1}^{\prime},J_{3}^{\prime}\right),$ $%
\left(\mathcal{K}_{2}^{\prime},J_{4}^{\prime}\right) )$ are two minimal
Stinespring representations for $\Phi $. Let $U_{1}$ and $U_{2}$ the unitary
operators defined in Proposition \ref{prop1}. Then for any $t\in \mathcal{G}%
,a_{1},\cdots,a_{n}\in \mathcal{A}$, $\xi_{1},\cdots,\xi_{n}\in \mathcal{H}%
_{1}$, we have
\begin{align*}
U_{1}v_{t}U_{1}^{\#}\left(\sum_{i=1}^{n}
\pi_{\varphi}^{\prime}\left(a_{i}\right) V_{\Phi}^{\prime}\xi_{i}\right)
&=U_{1}\left(\sum_{i=1}^{n}
\pi_{\varphi}\left(\beta_{t}^{\eta}\left(a_{i}\right)\right)
v_{t}V_{\Phi}\xi_{i}\right) \\
&=U_{1}\left(\sum_{i=1}^{n}\pi_{\varphi}\left(\beta_{t}^{\eta
}\left(a_{i}\right) \right) V_{\Phi}u_{t}\xi_{i}\right) \\
&=\sum_{i=1}^{n}w_{t}\pi_{\varphi}^{\prime}\left(a_{i}\right)
w_{t}^{\#}V_{\Phi}^{\prime}u_{t}\xi_{i} \\
&=w_{t}\left(\sum_{i=1}^{n}\pi_{\varphi}^{\prime}\left(a_{i}\right)
V_{\Phi}^{\prime}\xi_{i}\right),
\end{align*}
which implies that $U_{1}v_{t}U_{1}^{\#}=w_{t}$ for all $t\in \mathcal{G}$.
On the other hand, since $J_4=\mathrm{id}_{\mathcal{K}_2}$, from \eqref{m5}
we obtain that
\begin{align*}
v_{t}^{\prime}U_{2}^{\#}\left(J_{2}^{\prime}\pi_{\mathcal{X}%
}^{\prime}\left(x\right) V_{\Phi}^{\prime}\xi\right)
&=v_{t}^{\prime}\left(\pi_{\mathcal{X}}\left(x\right) V_{\Phi}\xi\right)
=v_{t}^{\prime}\left(\Phi(x)\xi\right)\\
& =\Phi(\eta_t(x))u_t\xi =\pi_{%
\mathcal{X}}\left(\eta_t(x)\right) V_{\Phi}u_t\xi
\end{align*}
and so from \eqref{eqn:pi-X and pi-varphi} for any $%
x_{1}, \cdots, x_{n}, y_{1}, \cdots, y_{m} \in \mathcal{X}$, $%
\xi_{1}, \cdots, \xi_{n}, \zeta_{1}, \cdots, \zeta_{n} \in \mathcal{H}_{1}$ and $t\in \mathcal{G}$ we
obtain that
\begin{align*}
&\hspace{-2.5cm}\left\langle
U_{2}v_{t}^{\prime}U_{2}^{\#}\left(J_{2}^{\prime}\sum_{i=1}^{n}\pi_{%
\mathcal{X}}^{\prime}\left(x_{i}\right)
V_{\Phi}^{\prime}\xi_{i}\right),\sum_{j=1}^{m}\pi_{\mathcal{X}%
}^{\prime}\left(y_{j}\right) V_{\Phi}^{\prime}\zeta _{j}\right\rangle \\
&=\sum_{i=1}^{n}\sum_{j=1}^{m} \left\langle J_3\pi_{%
\mathcal{X}}\left(y_{j}\right)^{\#}\pi_{\mathcal{X}}\left(\eta_{t}%
\left(x_{i}\right)\right) V_{\Phi}u_{t}\xi_{i},V_{\Phi}\zeta_{j}\right\rangle
\\
&=\sum_{i=1}^{n}\sum_{j=1}^{m} \left\langle
J_3\pi_{\varphi}\left(\left\langle y_{j},\eta_{t}\left(x_{i}\right)
\right\rangle\right) V_{\Phi}u_{t}\xi_{i},V_{\Phi}\zeta_{j}\right\rangle \\
&=\sum_{i=1}^{n}\sum_{j=1}^{m} \left\langle
V_{\Phi}^{\#}\pi_{\varphi}\left(\left\langle
y_{j},\eta_{t}(x_{i})\right\rangle\right)
V_{\Phi}u_{t}\xi_{i},J_1\zeta_{j}\right\rangle \\
&=\sum_{i=1}^{n}\sum_{j=1}^{m} \left\langle
\varphi\left(\left\langle y_{j},\eta_{t}\left(x_{i}\right)
\right\rangle\right) u_{t}\xi_{i},J_1\zeta_{j}\right\rangle \\
&=\sum_{i=1}^{n}\sum_{j=1}^{m} \left\langle \pi_{\mathcal{X}%
}^{\prime}\left(y_{j}\right)^{\#}\pi_{\mathcal{X}}^{\prime}\left(\eta_{t}%
\left(x_{i}\right) \right) V_{\Phi}^{\prime}u_{t}\xi_{i},
V_{\Phi}^{\prime}J_1\zeta_{j}\right\rangle \\
&=\left\langle w_{t}^{\prime}\sum_{i=1}^{n}J_{2}^{\prime}\pi_{%
\mathcal{X}}^{\prime}\left(x_{i}\right) V_{\Phi}\xi_{i},
\sum_{j=1}^{m}\pi_{\mathcal{X}}^{\prime}\left(y_{j}\right)
V_{\Phi}^{\prime}\zeta_{j}\right\rangle
\end{align*}
and then taking into account that $\left[ J_{2}^{\prime}\pi_{\mathcal{X}%
}^{\prime}\left(\mathcal{X}\right) V_{\Phi}^{\prime}\mathcal{H}_{1}\right]
=J_{2}^{\prime}\left[ \pi_{\mathcal{X}}^{\prime}\left(\mathcal{X}\right)
V_{\Phi}^{\prime}\mathcal{H}_{1}\right] =\mathcal{K}_{2}^{\prime }, $ we
deduce that $U_{2}v_{t}^{\prime}U_{2}^{\#}=w_{t}^{\prime}$ for all $t\in
\mathcal{G}$.
\end{proof}

Let $\mathcal{G}$ be a locally compact group with a left Haar measure $dt$
and the modular function $\Delta :\mathcal{G}\to (0,\infty )$. Let $%
\mathcal{X}$ be a full Hilbert $C^*$-module over a unital $C^*$%
-algebra $\mathcal{A}$ and $\eta $ an action of $\mathcal{G}$ on $\mathcal{X}
$.

The linear space $C_{c}(\mathcal{G},\mathcal{X})$ of all continuous
functions from $\mathcal{G}$ to $\mathcal{X}$ with compact support has a
structure of pre-Hilbert $\mathcal{G}\times_{\beta^{\eta}}\mathcal{A}$%
-module with the action of $\mathcal{G}\times_{\beta^{\eta}}\mathcal{A}$ on $%
C_{c}(\mathcal{G},\mathcal{X})$ given by
\begin{equation*}
\left(\widehat{x}f\right) \left(s\right) =\int_{\mathcal{G}}\widehat{%
x}\left(t\right) \beta_{t}^{\eta}\left(f\left(t^{-1}s\right) \right) dt
\end{equation*}
for all $\widehat{x}\in C_{c}(\mathcal{G},\mathcal{X})$, $f\in C_{c}(%
\mathcal{G},\mathcal{A})$ and the inner product given by
\begin{equation} \label{eqn:inner product of closed product}
\left\langle \widehat{x},\widehat{y}\right\rangle \left(s\right)
=\int_{\mathcal{G}}\beta_{t^{-1}}^{\eta}\left(\left\langle \widehat{x%
}(t),\widehat{y}\left(ts\right) \right\rangle \right) dt.
\end{equation}
The crossed product of $\mathcal{X}$ by $\eta $, denoted by $\mathcal{G}%
\times_{\eta}\mathcal{X}$, is the Hilbert $\mathcal{G}\times_{\beta ^{\eta}}%
\mathcal{A}$-module obtained by the completion of the pre-Hilbert $\mathcal{G%
}\times_{\beta^{\eta}}\mathcal{A}$-module $C_{c}(\mathcal{G},\mathcal{X})$
(see, for example, \cite{K}).

Suppose that $\varphi :\mathcal{A}\to \mathcal{L}(\mathcal{H}_{1})$
is an $\alpha $-CP map such that the constant $M(a)$ from Definition \ref%
{def4} (iii) is of the form $K(a)\left\Vert a\right\Vert $ with $K(a)>0.$

If $\varphi $ is $u$-covariant with respect to $\beta ^{\eta }$, then there
is a unique $\widetilde{\alpha }$-CP map $\widetilde{\varphi }:\mathcal{G}%
\times _{\beta ^{\eta }}\mathcal{A}\to \mathcal{L}(\mathcal{H}_{1})$
such that%
\begin{equation*}
\widetilde{\varphi }\left( f\right) =\int_{\mathcal{G}}\varphi
\left( f(t)\right) u_{t}dt
\end{equation*}%
for all $f\in C_{c}(\mathcal{G},\mathcal{A})$, where $\widetilde{\alpha }$ $%
\left( f\right) =\alpha \circ f$ for all $f\in C_{c}(\mathcal{G},\mathcal{A}%
) $ (it is similar to \cite[Theorem 4.3]{HJ}). Moreover, if $\left( \pi
_{\varphi },V_{\Phi },v_{t},\left( \mathcal{K}_{1},J_{3}\right) \right) $ is
the minimal KSGNS construction associated to $\varphi $, then $\left(
\widehat{\pi _{\varphi }},V_{\Phi },\left( \mathcal{K}_{1},J_{3}\right)
\right) $ is the minimal KSGNS construction associated to $%
\widetilde{\varphi }$, where $\widehat{\pi _{\varphi }}\left( f\right)
=\int_{\mathcal{G}}\pi _{\varphi }\left( f(t)\right) v_{t}dt$ for
all $f\in C_{c}(\mathcal{G},\mathcal{A})$. The next result may be compared with
\cite[Corollary 4.3]{HJK}.

\begin{thm}
Let $\mathcal{X}$ be a Hilbert $\mathcal{A}$-module, $\Phi :\mathcal{X}\to \mathcal{L}(\mathcal{H}_{1},\mathcal{H}%
_{2})$ be a $\varphi$-map for an $\alpha$-CP map $\varphi:\mathcal{A}\mapsto%
\mathcal{H}_1$ and $\left(u^{\prime},u\right)$-covariant with respect to $%
\eta $. Then there exist a $\widehat{\pi_{\varphi}}$-representation $%
\widehat{\pi_{\mathcal{X}}}$ of $\mathcal{G}\times_{\eta}\mathcal{X}$ on the
Krein spaces $\left(\mathcal{K}_{1},J_{3}\right) $ and $\left(\mathcal{K}%
_{2},J_{4}\right)$, and a unique $\widetilde{\varphi}$-map $\widetilde{\Phi}%
: \mathcal{G}\times _{\eta}\mathcal{X}\to \mathcal{L}(\mathcal{H}%
_{1},\mathcal{H}_{2})$ for the $\widetilde{\alpha}$-CP map $\widetilde{%
\varphi}$ such that
\begin{eqnarray*}
\widetilde{\Phi}\left(\widehat{x}\right) =W_{\Phi}^{\#}\widehat{\pi_{%
\mathcal{X}}}\left(\widehat{x}\right) V_{\Phi} =\int_{\mathcal{G}%
}\Phi\left(\widehat{x}\left(t\right) )\right) u_{t}dt
\end{eqnarray*}
for all $\widehat{x}\in C_{c}(\mathcal{G},\mathcal{X})$.
\end{thm}

\begin{proof}
Let $\left( \pi _{\mathcal{X}},\pi _{\varphi },V_{\Phi },W_{\Phi
},v,v^{\prime },\left( \mathcal{K}_{1},J_{3}\right) ,\left( \mathcal{K}%
_{2},J_{4}\right) \right) $ be the covariant KSGNS construction associated
to $\Phi $. Let $\widehat{x}\in C_{c}(\mathcal{G},\mathcal{X})$. Since
\begin{align*}
\left\vert \left\langle \int_{\mathcal{G}}\pi _{\mathcal{X}}\left(
\widehat{x}\left( t\right) \right) v_{t}\xi dt,\zeta \right\rangle
\right\vert & =\int_{\mathcal{G}}\left\vert \left\langle \pi _{%
\mathcal{X}}\left( \widehat{x}\left( t\right) \right) v_{t}\xi ,\zeta
\right\rangle \right\vert dt\\
&\leq C(\widehat{x})\int_{\mathcal{G}%
}\left\Vert \widehat{x}\left( t\right) \right\Vert \left\Vert \xi
\right\Vert \left\Vert \zeta \right\Vert dt \\
& \leq C(\widehat{x})\left\Vert \xi \right\Vert \left\Vert \zeta \right\Vert
\left\Vert \widehat{x}\right\Vert _{1}
\end{align*}%
for all $\xi \in \mathcal{K}_{1}$, $\zeta \in \mathcal{K}_{2}$ and for some
constant $C(\widehat{x})>0$, and also for any $x,y\in \mathcal{X}$, $a\in %
\mathcal{A}$ and $\xi ,\zeta \in \mathcal{H}_{1}$, we obtain that $\pi _{%
\mathcal{X}}(\widehat{x}(t))v_{t}(a\otimes \xi )=\Phi (\widehat{x}(t)\beta
_{t}^{\eta }(a))u_{t}\xi $ and
\begin{align*}
\Phi (\widehat{x}(t)\beta _{t}^{\eta }(a))^*\Phi (y)& =J_{1}\Phi (%
\widehat{x}(t)\beta _{t}^{\eta }(a))^{\#}\Phi (y)\\
&=J_{1}\varphi (\langle
\widehat{x}(t)\beta _{t}^{\eta }(a),y\rangle ) \\
& =J_{1}\varphi (\beta _{t}^{\eta }(a)^{\#}\alpha (\langle \widehat{x}%
(t),y\rangle ))\\
&=V_{\Phi }^{\#}\pi _{\varphi }(\beta _{t}^{\eta }(a))^*\pi _{\varphi }(\alpha (\langle \widehat{x}(t),y\rangle ))
\end{align*}%
and so we obtain
\begin{align*}
\left\langle \pi _{\mathcal{X}}(\widehat{x}(t))v_{t}(a\otimes \xi ),\Phi
(y)\zeta \right\rangle & =\left\langle u_{t}\xi ,\Phi (\widehat{x}(t)\beta
_{t}^{\eta }(a))^*\Phi (y)\zeta \right\rangle \\
& =\left\langle u_{t}J_{1}\xi ,J_{1}V_{\Phi }^{\#}\pi _{\varphi }(\beta
_{t}^{\eta }(a))^*\pi _{\varphi }(\alpha (\langle \widehat{x}%
(t),y\rangle ))\zeta \right\rangle \\
& =\left\langle a\otimes \xi ,v_{t}^*\pi _{\varphi }(\alpha (\langle
\widehat{x}(t),y\rangle ))\zeta \right\rangle .
\end{align*}%
Therefore, there exists an element $\widehat{\pi _{\mathcal{X}}}\left(
\widehat{x}\right) \in \mathcal{L}(\mathcal{K}_{1},\mathcal{K}_{2})$ such
that
\begin{equation*}
\widehat{\pi _{\mathcal{X}}}\left( \widehat{x}\right) =\int_{%
\mathcal{G}}\pi _{\mathcal{X}}\left( \widehat{x}\left( t\right) \right)
v_{t}dt.
\end{equation*}%
In this way, we obtain a map $\widehat{\pi _{\mathcal{X}}}:C_{c}(\mathcal{G},%
\mathcal{X})\to \mathcal{L}(\mathcal{K}_{1},\mathcal{K}_{2})$ being
extended by continuity to a continuous map $\widehat{\pi _{\mathcal{X}}}:%
\mathcal{G}\times _{\eta }\mathcal{X}\to \mathcal{L}(\mathcal{K}_{1},%
\mathcal{K}_{2})$. Moreover, by applying
\eqref{eqn:inner product of closed
product}, the covariance property of $\pi _{\varphi }$ and the Fubini's
theorem, we observe that
\begin{align*}
\widehat{\pi _{\varphi }}\left( \left\langle \widehat{x},\widehat{y}%
\right\rangle \right) & =\int_{\mathcal{G}}\pi _{\varphi }\left(
\left\langle \left\langle \widehat{x},\widehat{y}\right\rangle \left(
t\right) \right\rangle \right) v_{t}dt=\int_{\mathcal{G}%
}\int_{\mathcal{G}}\pi _{\varphi }\left( \beta _{s^{-1}}^{\eta
}\left( \left\langle \widehat{x}(s),\widehat{y}\left( st\right)
\right\rangle \right) \right) v_{t}dsdt \\
& =\int_{\mathcal{G}}\int_{\mathcal{G}}v_{s^{-1}}\pi
_{\varphi }\left( \left\langle \widehat{x}(s),\widehat{y}\left( st\right)
\right\rangle \right) \left( v_{s^{-1}}\right) ^{\#}v_{t}dsdt \\
& =\int_{\mathcal{G}}\int_{\mathcal{G}}v_{s}^*\pi _{%
\mathcal{X}}\left( \widehat{x}(s)\right) ^{\#}\pi _{\mathcal{X}}\left(
\widehat{y}\left( st\right) \right) v_{st}dsdt \\
& =\int_{\mathcal{G}}\int_{\mathcal{G}}J_{3}v_{s}^*\pi
_{\mathcal{X}}\left( \widehat{x}(s)\right)^*J_{4}\pi _{\mathcal{X}%
}\left( \widehat{y}\left( g\right) \right) v_{g}dgds \\
& =J_{3}\widehat{\pi _{\mathcal{X}}}\left( \widehat{x}\right)^*J_{4}%
\widehat{\pi _{\mathcal{X}}}\left( \widehat{y}\right) \\
& =\widehat{\pi _{\mathcal{X}}}\left( \widehat{x}\right) ^{\#}\widehat{\pi _{%
\mathcal{X}}}\left( \widehat{y}\right)
\end{align*}%
for all $\widehat{x},\widehat{y}\in C_{c}(\mathcal{G},\mathcal{X})$, $%
\widehat{\pi _{\mathcal{X}}}$ is a $\widehat{\pi _{\varphi }}$-representation of $\mathcal{G}\times _{\eta }\mathcal{X}$ on the Krein spaces $\left( \mathcal{K}_{1},J_{1}\right) $ and $\left( \mathcal{K}%
_{2},J_{2}\right) $. Consider the map $\widetilde{\Phi }:\mathcal{G}\times
_{\eta }\mathcal{X}\to \mathcal{L}(\mathcal{H}_{1},\mathcal{H}_{2})$
defined by
\begin{equation*}
\widetilde{\Phi }\left( z\right) =W_{\Phi }^{\#}\widehat{\pi _{\mathcal{X}}}%
\left( z\right) V_{\Phi }.
\end{equation*}%
Then for all $\widehat{x}\in C_{c}(\mathcal{G},\mathcal{X})$ we have
\begin{align*}
\widetilde{\Phi }\left( \widehat{x}\right) & =W_{\Phi }^{\#}\widehat{\pi _{%
\mathcal{X}}}\left( \widehat{x}\right) V_{\Phi }=\int_{\mathcal{G}%
}W_{\Phi }^{\#}\pi _{\mathcal{X}}\left( \widehat{x}\left( t\right) \right)
v_{t}V_{\Phi }dt=\int_{\mathcal{G}}W_{\Phi }^{\#}\pi _{\mathcal{X}%
}\left( \widehat{x}\left( t\right) \right) V_{\Phi }u_{t}dt \\
& =\int_{\mathcal{G}}\Phi \left( \widehat{x}\left( t\right) \right)
u_{t}dt.
\end{align*}%
on the other hand, for any $z_{1},z_{2}\in \mathcal{G}\times _{\eta }%
\mathcal{X}$, we obtain that
\begin{align*}
\widetilde{\Phi }\left( z_{1}\right) ^{\#}\widetilde{\Phi }\left(
z_{2}\right) & =V_{\Phi }^{\#}\widehat{\pi _{\mathcal{X}}}\left(
z_{1}\right) ^{\#}W_{\Phi }W_{\Phi }^{\#}\widehat{\pi _{\mathcal{X}}}\left(
z_{2}\right) V_{\Phi }=V_{\Phi }^{\#}\widehat{\pi _{\mathcal{X}}}\left(
z_{1}\right) ^{\#}\widehat{\pi _{\mathcal{X}}}\left( z_{2}\right) V_{\Phi }
\\
& =V_{\Phi }^{\#}\widehat{\pi _{\varphi }}\left( \left\langle
z_{1},z_{2}\right\rangle \right) V_{\Phi }=\int_{\mathcal{G}}V_{\Phi
}^{\#}\pi _{\varphi }\left( \left\langle z_{1},z_{2}\right\rangle (t)\right)
v_{t}V_{\Phi }dt \\
& =\int_{\mathcal{G}}V_{\Phi }^{\#}\pi _{\varphi }\left( \left\langle
z_{1},z_{2}\right\rangle (t)\right) V_{\Phi }u_{t}dt=\int_{\mathcal{G}%
}\varphi \left( \left\langle z_{1},z_{2}\right\rangle (t)\right) u_{t}dt \\
& =\widetilde{\varphi }\left( \left\langle z_{1},z_{2}\right\rangle \right)
\end{align*}%
and so we conclude that $\widetilde{\Phi }$ is a $\widetilde{\varphi }$-map.
\end{proof}


\begin{rem}
Suppose that $\left( \pi _{\mathcal{X}},\pi _{\varphi },V_{\Phi },W_{\Phi
},v,v^{\prime },\left( \mathcal{K}_{1},J_{3}\right) ,\left( \mathcal{K}%
_{2},J_{3}\right) \right) $ is the minimal covariant KSGNS construction
associated to $\Phi $. Then one can easily conclude that $\left( \widehat{\pi _{\mathcal{X}}},V_{\Phi
},W_{\Phi },\left( \mathcal{K}_{1},J_{3}\right) ,\left( \mathcal{K}%
_{2},J_{4}\right) \right) $ is the minimal KSGNS construction associated to $%
\widetilde{\Phi }$.
\end{rem}

\bigskip


\end{document}